%% file: DLvar.tex
\documentclass[14pt,reqno, a4paper]{amsart}

\usepackage[utf8]{inputenc}
\usepackage{mathscinet}
\usepackage[T1]{fontenc}
\usepackage{amssymb}
\usepackage{amscd}
\usepackage{amsxtra}
\usepackage{amsmath}
\usepackage{amsfonts}
\usepackage{amsthm}
\usepackage{calligra}
\usepackage{comment}
\usepackage{csquotes}
\usepackage{dsfont}
\usepackage{graphicx}
\usepackage{enumitem}
\usepackage{placeins}
\usepackage{mathrsfs}
\usepackage[english]{babel} 
\usepackage{hyperref}  
\usepackage{mathtools}
\usepackage{stackrel,amssymb}
\usepackage{tikz-cd}
\usepackage[backend=biber,style=alphabetic,sorting=nyt,maxbibnames=5,doi=false,isbn=false,url=false]{biblatex} 
\DeclareLabelalphaTemplate[]{
  \labelelement{
    \field[final]{shorthand}
    \field{label}
    \field[strwidth=3,strside=left,ifnames=1]{labelname}
    \field[strwidth=1,strside=left]{labelname}
  }
}
\DeclareFieldFormat{extraalpha}{#1}
\newtheorem{theorem}{Theorem}[section]
\newtheorem*{theorem*}{Theorem}
\newtheorem{lemma}[theorem]{Lemma}
\newtheorem{proposition}[theorem]{Proposition}
\newtheorem{corollary}[theorem]{Corollary}
\newtheorem* {cor*}{Corollary}
\newtheorem*{corollary*}{Corollary}
\newtheorem*{proposition*}{Proposition}
\newtheorem*{lemma*}{Lemma}

\newtheorem{lemdfn}[theorem]{Lemma/Definition}

\theoremstyle{definition}
\newtheorem{dfn}[theorem]{Definition}
\newtheorem{example}[theorem]{Example}

\theoremstyle{remark}
\newtheorem{remark}{Remark}
\newtheorem*{remark*}{Remark}

\newtheorem*{note*}{Note}

\newcommand{\Fp}{\mathbb{F}_p}
\newcommand{\Fq}{\mathbb{F}_q}
\newcommand{\Fpbar}{\overline{\mathbb{F}}_p}

\renewcommand{\P}{\mathbb{P}}

\newcommand{\proj}{\text{pr}}
\newcommand{\Osh}{\mathcal{O}}

\newcommand{\GL}{{\rm GL}}
\newcommand{\id}{\text{id}}

\newcommand{\Xo}{\overline{X}}

\newcommand{\leftp}{\left(}
\newcommand{\rightp}{\right)}

\newcommand{\ZZ}{{\mathbb{Z}}}
\newcommand{\Het}{H_{\normalfont\text{\'{e}t}}}

\newcommand{\Ind}{{\text{\normalfont ind}}}
\newcommand{\et}{{\normalfont\text{\'{e}t}}}

\newcommand{\Zar}{{\normalfont\text{Zar}}}
\newcommand{\supp}{{\normalfont\text{supp}}}
\newcommand{\K}{{\normalfont \text{Ker}}}

\newcommand{\St}{{\normalfont \text{St}}}

\DeclareMathOperator{\sheafhom}{\mathscr{H}\text{\kern -3pt {\textit{om}}}}
\DeclareMathOperator{\sheafdif}{\mathscr{D}\text{\kern -1pt {\textit{if}}}}

\title{Cohomology and geometry of Deligne--Lusztig varieties for $\GL_n$}
\author{Yingying Wang}

\begin{document}
\maketitle
	\input{abstract}

	\tableofcontents
	\input{introduction}

	\input{body}

	\printbibliography
	\medskip
	\textsc{ Universit\"{a}t Duisburg-Essen, 45127 Essen, Germany}
\smallskip
  \textit{E-mail address:}
\email{yingying.wang@uni-due.de}
\end{document}

%% file: abstract.tex
\begin{abstract}
  We give a description of the cohomology groups of the structure sheaf on smooth compactifications $\Xo(w)$ of Deligne--Lusztig varieties $X(w)$ for $\GL_n$, for all elements $w$ in the Weyl group. 
  As a consequence, we obtain the ${\rm mod}\ p^m$ and integral $p$-adic \'{e}tale cohomology of $\Xo(w)$. 
 Moreover, using our result for $\Xo(w)$  and  a spectral sequence associated to a stratification of $\Xo(w)$, we deduce 
the  ${\rm mod}\ p^m$ and integral $p$-adic \'{e}tale cohomology with compact support of $X(w)$.

 In our proof of the main theorem, in addition to considering the Demazure--Hansen smooth compactifications of $X(w)$, we show that a similar class of constructions provide smooth compactifications of $X(w)$ in the case of $\GL_n$. Furthermore, we show in the appendix that the Zariski closure of $X(w)$, for any connected reductive group $G$ and any $w$, has pseudo-rational singularities. 
 
	\end{abstract}

%% file: introduction.tex
\section{Introduction}
Let $G/\Fpbar$ be a connected reductive group defined over a finite field $\Fq$ and let $F$ be the Frobenius endomorphism. Deligne--Lusztig varieties were introduced in \cite{DL} for studying irreducible representations of $G(\Fq)$. The aim of this paper is to describe the cohomology groups of the structure sheaf for Deligne--Lusztig varieties for $\GL_n$ and their interpretations as $\GL_n(\Fq)$-representations. 
\subsection{General Background}

Fix a maximal torus $T^*$ and a Borel subgroup $B^*$ containing $T^*$ in $G$. 
A Deligne--Lusztig variety $X(w)$ is a locally closed subscheme of $G/B^*$ consisting of Borel subgroups $B$ whose relative position with $FB$
 is given by an element $w$ of the Weyl group $W$. As $\Fpbar$-schemes,  Deligne--Lusztig varieties are quasi-projective and smooth of dimension $l(w)$, which is the Bruhat length of $w$.

Deligne and Lusztig considered the virtual representations arising from the $\ell$-adic cohomology with compact support of $X(w)$ and their \'{e}tale coverings for $\ell\neq p$. They showed that any irreducible representation of $\GL_n(\Fq)$ is contained in one of such virtual representations \cite[\textsection 7]{DL}. In the same paper, 
they constructed smooth compactifications $\Xo(w)$ of $X(w)$ for each reduced expression of $w$ \cite[\textsection 9]{DL}, which are analogous to the Demazure--Hansen desingularization of Schubert varieties.

 \medskip

A prominent example of Deligne--Lusztig varieties in the case of $G=\GL_n$ is $X(\mathbf{w})$, where $\mathbf{w}$ corresponds to the standard Coxeter element. It is isomorphic to the complement of all $\Fq$-rational hyperplanes in the projective space $\mathbb{P}^{n-1}_{\Fpbar}$ (resp. $\mathbb{P}^{n-1}_{\Fq}$), which we denote by $\mathcal{X}^{n-1}_{\Fpbar}$ (resp. $\mathcal{X}^{n-1}_{\Fq}$.). 
This example was defined by Deligne and Lusztig \cite{DL}, which first appeared in \cite[Example 1]{Lusztig75}. 
In \cite{GK}, Gro{\ss}e-Kl{\"{o}}nne gave a vanishing result for the cohomology of sheaves of logarithmic differential forms on a smooth compactification of $\mathcal{X}^{n-1}_{\Fq}$. 

 	For $K$ a nonarchimedean local field, one may take the complement $\mathcal{X}^{n-1}_{\mathbb C_p}$ (resp. $\mathcal{X}^{n-1}_{K}$) 
 of all $K$-rational hyperplanes in the rigid analytic projective space $\mathbb P^{n-1}_{\mathbb C_p}$ (resp. $\mathbb P^{n-1}_{K}$). This space is preserved under $\GL_n(K)$-action and is an admissible open of the rigid analytic projective space. It was introduced by Drinfeld \cite[\textsection 6]{Drinfeld74} and often referred to as the $p$-adic Drinfeld half space when $K$ is a finite extension of $\mathbb Q_p$.  
 	
 	The $p$-adic Drinfeld half space admits a semistable weak formal model whose generic fibre recovers the $p$-adic Drinfeld half space. Moreover, 
 	a smooth compactification of $\mathcal{X}^{n-1}_{\Fq}$ lives in the special fibre of this weak formal model as 
 	an irreducible component cf. \cite{GK}, \cite{CDN}. The cohomology theories for certain coherent sheaves of the $p$-adic Drinfeld half space have been studied by Schneider and Stuhler \cite{ScSt}, Schneider and Teitelbaum \cite{ST97}, Gro\ss e-Kl\"{o}nne \cite{GK}, Orlik \cite{Orlik08}, and many others. 

\subsection{The individual $\ell$-adic cohomology for Deligne--Lusztig varieties}

In \cite{Orlik18}, Orlik provided a strategy for computing the $\ell$-adic cohomology groups with compact support of Deligne--Lusztig varieties for $G=\GL_{n}$. We briefly recall the methods of  loc. cit. which are relevant to us. 

Let $F^+$ be the free monoid generated by a fixed set of standard generators  of $W$. The generalized Deligne--Lusztig variety attached to $w\in F^+$ were introduced in \cite{Lusztig78}. It is a subscheme of $(G/B^*)^{r+1}$  with  relative positions of successive pairs of Borel subgroups determined by the standard generators in the expression of $w$. 
For the precise definition see Section \ref{section2}.

 Orlik 
studied the association between cohomology groups
 of $\Xo(w)$ and $\Xo(w')$, where $w,w'\in F^+$ and
 $w'$ is obtained via replacing a reduced subexpression of $w$ by another according to a relation in the group presentation of $W$ or via shifting the first standard generator in the expression of $w$ to the end. In particular, he constructed three operators $C, K, R$ on elements of $F^+$. As the construction of the operators is related to the relations in the group presentation of $W$ cf. Section \ref{steps}, one can always use finitely many of these operators to transform $w$ into either a standard Coxeter element in a standard Levi subgroup of $G$, or of the form $svs$, when $G=\GL_n$ \cite{GP}. 
 
 If we have $svs\in F^+$, then $\Xo(svs)\to \Xo(vs)$ is a $\mathbb P^1$-fibration. This let one relate the cohomology group of $\Xo(svs)$ with the one for $\Xo(vs)$, which is of dimension $\dim_{\Fpbar}(\Xo(svs))-1$. 
 
  Orlik then establishes a double induction procedure with respect to the length of $w$ and the number of operators applied. The base case is the cohomology of $\Xo(w)$ with $w\in F^+$ having no repeating terms in its expression. 
  As $\Xo(w)$ is universally homeomorphic to a finite disjoint union of products of smooth compactifications of $\mathcal{X}^{j}_{\Fpbar}$ with $j\leq n-1$, its $\ell$-adic \'{e}tale cohomology can be determined.

\subsection{Statement of results}
We adapt Orlik's method of double induction to the cohomology of the structure sheaf on $\Xo(w), w\in F^+$.
The cohomology of the structure sheaf for the smooth compactification $\Xo(s_1\cdots s_{n-1})$ of $X(\mathbf w)$ follows from a result of Gro\ss e-Kl\"{o}nne for the smooth compactification of $\mathcal{X}^{n-1}_{\Fq}$ after base changing to $\Fpbar$ {\cite[Theorem 2.3]{GK}}:

 		 \begin{theorem*}[Gro{\ss}e-Kl{\"{o}}nne]
 	Let $G=\GL_n$ and $\mathbf w =s_1\cdots s_{n-1}$ the standard Coxeter element. Then
 		\[\arraycolsep=2pt\def\arraystretch{1.2}
H^k\left(\Xo(s_1\cdots s_{n-1}), \Osh_{\Xo(s_1\cdots s_{n-1})}\right)=\left\{
\begin{array}{lc}
\Fpbar, & k=0,\\
0, & k>0.
\end{array}
\right.
\]
 \end{theorem*}

 	 Our base cases are $\Xo(w)$, when $w$ is a standard Coxeter element for a standard Levi subgroup $L$ of $\GL_n$. If $L$ is a proper standard Levi subgroup, then  $\Xo(w)$ is a finite disjoint union, with each irreducible component isomorphic to the product of smooth compactifications of $\mathcal{X}^j_{\Fpbar}$ of dimensions $j<n-1$. In particular, the number of irreducible components  $\Xo(w)$ is given by the number of $\Fq$-rational points on $G/P$, where $P$ is the standard parabolic subgroup associated to $L$. The cohomology of  $\Xo(w)$ then also follows from the above theorem. 
 	 
 	 The double induction strategy then allows us to deduce the following result: 
 	 
 	 	\begin{theorem}\label{main}
Let $G=\GL_n$ and $w\in F^+$ with $w=s_{i_1}\cdots s_{i_r}$. Let $I=\{s_{i_1},\ldots , s_{i_r}\}$ and $P_I=B^*W_IB^*$ be the standard parabolic subgroup associated to $I$, then 

	\[\arraycolsep=2pt\def\arraystretch{1.2}
H^k\left(\Xo(w), \Osh_{\Xo(w)}\right)=\left\{
\begin{array}{lc}
\Ind_{P_I(\Fq)}^{\GL_n(\Fq)} \mathds{1}_{\Fpbar}, & k=0,\\
0, & k>0,
\end{array}
\right.
\]
where $\mathds{1}_{\Fpbar}$ is the trivial $P_I(\Fq)$-representation with coefficients in $\Fpbar$.   
\end{theorem}

If $w=svs$, we construct maps $\Xo(svs)\to \Xo(sv)$ and $\Xo(svs)\to \Xo(vs)$ and show that they give isomorphisms of cohomology groups with respect to the structure sheaf. These isomorphisms are $\GL_n(\Fq)$-equivariant. Consequently, they also induce isomorphisms on cohomology groups in the case of the cyclic shift operator $C$ cf. Definition \ref{operatorC}. 
The construction of these maps only relies on the fact that $\GL_n$ is split reductive, so this step applies to any split reductive connected groups defined over $\Fq$.

For the operators $K$ and $R$ corresponding to relations in the symmetric group cf. Definition \ref{operatorsKR}, we first construct an intermediate proper smooth scheme such that both $\Xo(w)$ and $\Xo(K(w))$ (resp. $\Xo(R(w))$) have birational morphism to it, then we use a theorem of Chatzistamatiou and R\"{u}lling to get the isomorphism of cohomology groups. The theorem \cite[Theorem 3.2.8]{CR11} shows that for proper smooth schemes over a perfect field, birational morphisms induce cohomological equivalence for the structure sheaf and the canonical bundle. In Appendix \ref{AppendixA}, we show that  $\overline{X(w)}$ has pseudo-rational singularities. For schemes with such singularities, there is a generalization of Chatzistamatiou--R\"{u}lling's result by Kov\'{a}cs \cite[Theorem 8.13]{Kovacs}. Hence the statements on the operators $K$ and $R$ follows from this and it generalizes to arbitrary connected reductive groups whenever there is an operator replacing a reduced subexpression of $w$ according to relation in the group presentation of $W$ that does not reduce the Bruhat length.

Since all steps in our version of the double induction induce isomorphisms on the cohomology groups as $\Fpbar$-vector spaces, and we have higher vanishing for the base case, 
we may analyze the $\GL_n(\Fq)$-action after we obtain the cohomology. In particular, 
 the global sections of the irreducible components are isomorphic to the base field, so they have to be trivial representations of their corresponding $P(\Fq)$. 

\medskip
 If we consider the sheaf of Witt vectors $W_m\left(\Osh_{\Xo(w)}\right)$ of length $m$ on $\Xo(w)$, our theorem implies that 
\[
H^0\left(\Xo(w), W_m\left(\Osh_{\Xo(w)}\right)\right)= \Ind_{P_I(\Fq)}^{\GL_n(\Fq)} W_m\left(\Fpbar\right).\]
Using the Artin--Schreier--Witt sequence we obtain the cohomology of the constant sheaves $\mathbb Z/p^m\mathbb Z$ and  $\mathbb Z_p$ on  $\Xo(w)_\et$, shown in Proposition \ref{ZpnZ} and Corollary \ref{Zpxo} below. 

\begin{corollary}
	Let $G=\GL_n$, and $w\in F^+$  with $w=s_{i_1}\cdots s_{i_r}$. Let $R$ be $\mathbb Z/p^m\mathbb Z, m\geq 1,$ or $\mathbb Z_p$. Then
	\[\arraycolsep=2pt\def\arraystretch{1.2}
\Het^k\left(\Xo(w), R\right)=\left\{
\begin{array}{lc}
\Ind_{P_I(\Fq)}^{\GL_n(\Fq)} \mathds{1}_{R}, & k=0,\\
0, & k>0,
\end{array}
\right.
\]
where $ \mathds{1}_{R}$ is the trivial $P_I(\Fq)$-representation with coefficients in $R$.
\end{corollary}

The higher vanishing of these cohomology groups are obtained inductively via the long exact sequence associated to the mod $p$ short exact sequence for $\mathbb Z/p^m  \mathbb Z$. The result for $\mathbb Z_p$-coefficients follows after verifying the Mittag--Leffler condition. 

\medskip

There is a spectral sequence associated to a stratification of $\Xo(w)$:
\[
E_1^{i,j}= \bigoplus_{\substack{ u\preceq w\\ \ell(u)=\ell(w)-i}} H^j_\et\left( \Xo(u), \mathbb Z/p^m\mathbb Z\right)	\Rightarrow H^{i+j}_{\et, c}\left(X(w),\mathbb Z/p^m\mathbb Z\right).
\]
The corollary above implies that this spectral sequence degenerates at the $E_2$-page. In particular, except at the $0$-th term, $E_1^{\bullet,0}$ is quasi-isomorphic to the Solomon--Tits complex mod $p$ cf. \cite{Orlik18}. The Solomon--Tits complex  is a simplicial complex constructed from the group $\GL_n(\Fq)$, whose $0$-th homology is $\mathbb Z$ and highest non-vanishing homology realizes the Steinberg representation over $\mathbb Z$. 
 This method has been used to study the compactly supported $\ell$-adic cohomology of $X(w)$  \cite[\textsection 5, \textsection 7]{Orlik18}, c.f. \cite[Ch VII]{DOR}. Using this method, we obtain the following results on the compactly supported cohomology for  $X(w)_\et$ with $\mathbb Z/p^m\mathbb Z$ and  $\mathbb Z_p$-coefficients
in Theorem \ref{modpnX(w)} and Corollary \ref{ZpX(w)}.

\begin{theorem}
Let $G=\GL_n$ and $w\in F^+$. Let $L_I\supseteq T^*$ be the standard Levi subgroup of $\GL_n$ such that $P_I=U_I\rtimes L_I$, where $U_I$ is the unipotent radical of $P_I$.  Let $R$ be $\mathbb Z/p^m\mathbb Z, m\geq 1$ or $\mathbb Z_p$.
Then  
	\[\arraycolsep=2pt\def\arraystretch{1.2}
H_{\et, c}^k\left(X(w),R\right)=\left\{
\begin{array}{lc}
0, & k\neq \ell(w),\\
\Ind_{P_I(\Fq)}^{\GL_n(\Fq)}\normalfont\text{St}_{L_I}, & k=\ell(w),
\end{array}
\right.
\]
 where $\normalfont\text{St}_{L_I}$ is the Steinberg representation for $L_I$ with coefficients in $R$. In particular, when $I=S$, we have $H^{\ell(w)}_{\et, c}\left(X(w),R \right)=\normalfont\text{St}_{\GL_n}$. 
\end{theorem}

Appendix \ref{AppendixA} provides some insights on the geometry of $\overline{X(w)}$ for arbitrary connected reductive group $G$ defined over $\Fq$. We show that they are strongly $F$-regular and thus have pseudo-rational singularities. This result is achieved using the fact that $\overline{X(w)}$ are associated to Zariski closure of Bruhat cells (resp. Schubert varieties) via the Lang map. Thus $\overline{X(w)}$ are $F$-regular because the the Zariski closure of Bruhat cells are $F$-regular \cite{BrionThomsen}.
It then follows from Kov\'{a}cs' theorem that the cohomology of $\overline{X(w)}$ is isomorphic to that of $\Xo(w)$ for the structure sheaf and the canonical bundle \cite{Kovacs}.

In fact, Lauritzen, Raben-Pedersen and Thomsen showed that Schubert varieties are not only strongly $F$-regular, but also globally $F$-regular \cite{LRT}.
A natural question to ask in this direction is that whether $\overline{X(w)}$ have globally defined singularities.

\subsection{Structure of the paper}
 
 In Section \ref{dfn} and \ref{section2} we recall basic definitions and properties of (generalized) Deligne--Lusztig varieties. Section \ref{sws} treats the cases used for the induction on the length of $w$. The cases used for the induction on the presentation of $w$ are studied in Section \ref{steps}. Section \ref{section5} establishes the base case of our induction in the case of $\GL_n$. Section \ref{theorem1} contains the proof of the main theorem \ref{main} and corollaries for the mod $p^m$ and $\mathbb Z_p$-cohomology groups.  In Section \ref{corollaryX(w)}, we give the proof for Theorem \ref{modpnX(w)} and Corollary \ref{ZpX(w)} on the compactly supported mod $p^m$ and $\mathbb Z_p$-cohomology of $X(w)$. Appendix \ref{AppendixA} examines the locally defined singularities of $\overline{X(w)}$ for any connected reducitve group $G$ and an application of a theorem in \cite{Kovacs}.

\subsection{Acknowledgements}
I thank my advisor, Sascha Orlik, for suggesting projects related to Deligne--Lusztig varieties, and for his guidance on during my PhD. 
I would like to thank  Kay R\"{u}lling for explaining his work \cite{CR11} to me during the writing of my thesis and suggesting the paper \cite{Kovacs} to me at my defense, which contributed to Appendix \ref{AppendixA}. Moreover, I wish to thank Pierre Deligne for answering my questions about  constructions used in \cite{DL}.  
I want to express my gratitude to the anonymous referee for many helpful remarks. 
Many thanks to Georg Linden and Thomas Hudson for giving detailed comments on earlier drafts of this paper. 
Thanks to S\'{a}ndor Kov\'{a}cs, George Lusztig, and Stefan Schr\"{o}er for helpful comments and discussions.

%% file: body.tex
\section{Deligne--Lusztig varieties}\label{dfn}
In this section we fix notations and conventions for Deligne--Lusztig varieties in general, and then specifically for the case of $G=\GL_n$ in Section \ref{GLdfn}. We will conclude this section with some examples of Deligne--Lusztig varieties for $\GL_n$. 

\subsection{Notations}\label{Gdefn}
Let $p$ be a prime number, and $q=p^r, r\geq 1$. Fix an algebraic closure $\Fpbar$ of the finite field $\Fp$ that contains the finite field $\mathbb F_q$. 
Here we recall some basic notions from the theory of reductive groups. The standard references we have used are \cite{DeGa}, \cite{Jantzen}, \cite{Humphreys} and \cite{SGA31}.

Let $G$ be a \emph{reductive algebraic $\Fpbar$-group} defined via base change by a connected reductive $\Fq$-group $G_0$. Let $F$ denote the \emph{Frobenius endomorphism} on $G$, obtained by extension from the Frobenius endomorphism of $G_0$. Denote with $G^F$ the fixed points of $G$ by the Frobenius. 

Note that the datum of a maximal torus, a Borel subgroup, and the Weyl group is unique up to unique isomorphism (cf. \cite[\textsection 1.1]{DL}).
Fix an \emph{$F$-stable Borel subgroup} $B^*$ of $G$ and an \emph{$F$-stable maximal torus}  $T^*$ such that $T^*\subseteq B^*$.

Let $W:=N(T^*)/T^*$ be the \emph{Weyl group}, where $N(T^*)$ is the normalizer of $T^*$ in $G$. At the same time, $W$ is the Weyl group of the root system of $T^*$, which  contains a set of simple roots that is in bijection with a set $S$ of generators of $W$. In the literature, elements of $S$ are sometimes called elementary reflections or simple reflections. We denote by $\ell(w)$ the \emph{Bruhat length} of $w\in W$. It is the minimal number $r$ such that $w$ can be written as the product $w=s_{i_1}\cdots s_{i_r}$, where $s_{i_j}\in S$, $j=1,\ldots,r$. Here we call $s_{i_1}\cdots s_{i_r}$ a \emph{reduced expression} of $w$. The \emph{Bruhat order} $\leq $ on $W$ is defined by: $w\leq v$ whenever $w,v\in W$ have reduced expressions $w=s_{i_1}\cdots s_{i_r}$ and $v=t_{1}\cdots t_{k},\  t_1,\ldots, t_k\in S$ such that $1\leq i_1\leq \cdots\leq i_r\leq k$, and $s_{i_j}=t_{i_j}$ for all $j=1,\ldots,r$ cf. \cite[\textsection 1.2.4]{GP}.  

We say that $G$ is \emph{split} when the maximal torus $T^*$ fixed above is a split maximal torus and is defined over $\Fq$. In this case the map $F:W\to W$ induced by the Frobenius is the identity \cite[\textsection 7.1.3, \textsection 4.3.1]{DM}.

Denote the \emph{opposite Borel subgroup} by $B^+$, recall that we have decompositions
 \[
 B^*=U^*T^*=U^*\rtimes T^*\hspace{0,5cm} B^+=U^+T^*=U^+\rtimes T^*,
 \]
 where $U^*$ and $U^+$ are the unipotent radicals of $B^*$ and $B^+$ respectively. 

Recall that the quotient $G/B^*$ exists in the category of $\Fpbar$-schemes \cite[III \textsection 3.5.4]{DeGa}, and that $G/B^*$ is an integral, projective and smooth scheme \cite[\textsection II.13.3]{Jantzen}.  
Since $\Fpbar$ is algebraically closed,  
the $\Fpbar$-rational points on $G/B^*$ correspond bijectively to the elements in $G\left(\Fpbar\right)/B^*\left(\Fpbar\right)$. 
Thus by Borel fixed point theorem, we know that the Borel subgroups of $G$ correspond to the $\Fpbar$-rational points on $G/B^*$ and they are all conjugate to $B^*$. 
Throughout this text, let $X$ be the set of all Borel subgroups of $G$ on which $G$ acts by conjugation. In particular, via the set theoretic identification between $G/B^*$ and $X$ given by $gB^*\mapsto gB^*g^{-1}$, one obtains a $\Fpbar$-scheme structure on $X$ with an $G$-action such that the identification $G/B^*\cong X$ is $G$-equivariant.
By abuse of notation, we write $gB^*\in G/B^*$, $B\in X$ or $x\in X$ for $\Fpbar$-rational points on $G/B^*$ and $X$ respectively, when there is no ambiguity.    
\subsection{Basic Constructions}
There is a left $G$-action on the product $X\times X$ given by the diagonal action:
\[\begin{array}{rcl}
	G\times\left(X\times X\right) &\longrightarrow & X\times X\\
	\left(g,\left(B_1,B_2\right)\right) &\longmapsto & \left(gB_1g^{-1}, gB_2g^{-1}\right).
\end{array}
\]
 The quotient $G\backslash \left(X\times X\right)$ is in bijection with the Weyl group $W$ as a result of the Bruhat decomposition \cite[\textsection 2.11]{BT65}. For each $w\in W$, the orbit $O(w)$ is of the form $G.\left(B^*, \dot{w}B^*\dot{w}^{-1}\right)$, where $\dot{w}\in N(T^*)$ is a representative of $w$. In particular, the orbit corresponding to the identity $e\in W$ is $O(e)\cong X$. We refer to \cite[\textsection 1]{DL} for basic properties of the orbits $O(w)$. Since $F:G\to G$ induces an automorphism on $X=G/B$, let $\Gamma_F\subseteq X\times X$ be the graph of  $F:X\to X$.

\begin{dfn}
\normalfont
 The \emph{Deligne--Lusztig variety} $X(w)$ for $G$ corresponding to $w\in W$ is defined as the intersection in $X\times X$ of $O(w)$ and the graph of $F$.
\[
X(w):= O(w)\times_{\left(X\times X\right)}\Gamma_F.
\] 
\end{dfn}

\begin{remark}
Note that this intersection is transverse \cite[p. 107]{DL}. Moreover,
	 the set of (rational) points of $X(w)$ corresponds to the subset of  Borel subgroups $B$ in $X$ such that $B$ and $F(B)$ are in relative position $w$, i.e.,
 \begin{equation*}
 	X(w)=\{B\in X\vert(B,F(B))\in O(w)\}.
 \end{equation*}
 			Additionally,  $X(w)$ is  a subscheme of $X$ of dimension $\ell(w)$ that is locally closed and smooth.
	 Considered as a subscheme of $X$, the Deligne--Lusztig variety $X(w)$ is stable under the $G^F$-action. Thus we have a $G^F$-action on $X(w)$.
\end{remark}

Deligne--Lusztig varieties may alternatively be defined as follows  cf. \cite[Example 3.10 (c)]{Lusztig78}.
\begin{dfn}
\normalfont
	Let  $w\in W$ and $w=s_{i_1}\cdots s_{i_r}$ with $s_{i_j}\in S$ be a reduced expression. The \emph{Deligne--Lusztig variety} associated to this reduced expression  is defined as:
		\begin{equation*}
			X(s_{i_1},\ldots,s_{i_r}):= \left\{(B_0,\ldots,B_r)\in X^{r+1}\vert (B_{j-1}, B_j)\in O(s_{i_j}), j=1,\ldots,r, FB_0=B_r
	\right\}.
		\end{equation*}

\end{dfn}

\begin{remark}
The definition above is independent of the reduced expression of $w$  up to canonical isomorphisms \cite[Example 3.10 (c)]{Lusztig78} cf. \cite[p. 759]{DMR}.

	 \end{remark}

\subsection{Smooth compactifications} In general, the Zariski closure $\overline{X(w)}$ of $X(w)$ in $X$ is not smooth. For each reduced decomposition $w=s_{i_1}\cdots s_{i_r}$ of $w\in W$ with $s_{i_j}\in S$, $j=1,\ldots,r$, we have a smooth compactification $\overline{X}(w)$ of $X(w)$ with a normal crossing divisor at infinity \cite[\textsection 9.10]{DL} defined as follows:

\begin{dfn}\label{smoothcompactifi}
\normalfont
	Let $w\in W$, and let  $w=s_{i_1}\cdots s_{i_r}$ be a reduced expression with $s_{i_j}\in S$, $j=1,\ldots,r$. 	We define 	
		\[
	\overline{O}(s_{i_1},\ldots,s_{i_r}):=\left\{(B_0,\ldots,B_r)\in X^{r+1}\vert (B_{j-1},B_j)\in \overline{O(s_{i_j})}, j=1,\ldots,r \right\},
	\]
	where $\overline{O(s_{i_j})}=O(s_{i_j}){\dot\cup} O(e)$, and
	\[
	\overline{X}(s_{i_1},\ldots,s_{i_r}):=\overline{O}(s_{i_1},\ldots,s_{i_r})\times_{\left(X\times X\right)}\Gamma_F,
	\]
	where $\overline{O}(s_{i_1},\ldots,s_{i_r})\to X\times X$ is the projection map $(B_0,\ldots,B_r)\mapsto (B_0, B_r)$. 
	
	We also use the notation $\Xo(w)$ for $\overline{X}(s_{i_1},\ldots,s_{i_r})$ when the reduced expression $w=s_{i_1}\cdots s_{i_r}$ is specified. 
\end{dfn}

\begin{remark}
	
	  We may alternatively write
	  \begin{multline*}
	  	\overline{X}(s_{i_1},\ldots,s_{i_r})=\\
	  	\left\{(B_0,\ldots,B_r)\in X^{r+1}\vert (B_{j-1},B_j)\in \overline{O(s_{i_j})}, j=1,\ldots,r, FB_0=B_r\right\}.
	  \end{multline*}
\end{remark}

\subsection{Affiness and Irreducibility}
\subsubsection*{Affineness}

 (i) Let $G/\Fpbar$ be any connected reductive group that is defined over $\Fq$. Let $h$ be the Coxeter number of $G$, which is the Bruhat length of any Coxeter element of $W$. If $q>h$, then $X(w)$ is affine for any $w\in W$ \cite[Theorem 9.7]{DL}.

\noindent (ii) If $w\in W$ is a Coxeter element, then $X(w)$ is affine \cite[Corollary 2.8]{Lusztig}.

\subsubsection*{Irreducibility}
Whether a Deligne--Lusztig variety is irreducible is completely dependent on the support of the corresponding Weyl group element. 

(i) When $w\in W$ is a Coxeter element, then $X(w)$ is irreducible \cite[Proposition 4.8]{Lusztig}. For the non-split groups, we say that $w$ is a Coxeter element if $w$ is a product of elements of $S$ which come from distinct $F$-orbits.

	(ii) For $w\in W$, $X(w)$ is irreducible if $w$ is not contained in any proper $F$-stable parabolic subgroup of $W$.  There are many proofs for this statement. Here we refer to \cite{BR} and \cite{Goertz09}. 
	Note that for split reductive groups $G$, this equivalent of saying that $X(w)$ is irreducible if  the support of $w$ is $S$ (cf. Definition \ref{support}). 
	
	 In addition, $X(w)$ has the same number of irreducible components as its smooth compactification $\Xo(w)$ (when one fixes a reduced expression of $w$).

	(iii) Let $v,w\in W$ such that 
	\[
	\supp(v)=\supp(w),
	\] 
	then $X(w)$ and $X(v)$ have the same number of irreducible components. By \cite[Proposition 2.3.8]{DMR}, one may write down the irreducible components, as well as the number of irreducible components.

	\subsection{Fibrations over $X$}\label{fibrations}

We recall basic constructions and observations from \cite[1.2(b)]{DL}. In this section we use the notations from Section \ref{Gdefn}. We say a morphism of $\Fpbar$-schemes $f:Y_1\to Y_2$ is a bundle with fibre $E$, if $Y_2$ admits an open covering with respect to a (Grothendieck) topology $\tau$ (e.g. Zariski open covering, fppf open covering) such that all (closed) fibres of $f$ are isomorphic to $E$ and $f$ is locally trivial with respect to this covering. 

 Let $\pi:G\to G/B^*$ be the canonical projection map. The Zariski open covering $\{\dot{w}U^+B^*\}_{w\in W}$  of $G$ gives a Zariski open covering $\{\pi(\dot{w}U^+B^*)\}_{w\in W}$ of $G/B$ cf. \cite[II \textsection 1.10]{Jantzen}. 

\begin{lemma}
Let $s\in W$ be a simple reflection. The projection map $\proj_i:O(s)\to X, i=1,2$ is an
	 (Zariski) $\mathbb A^1$-bundle with respect to the open covering $\{\pi(\dot{w}U^+B^*)\}_{w\in W}$  of $X$.
	 	Moreover, $O(s)$ is isomorphic to the homogeneous $G$-space $G/(B^*\cap\dot s B^*\dot s)$ over $G/B^*$, whose fibres are isomorphic to $B^*/(B^*\cap \dot{s}B^*\dot{s})\cong \mathbb A^1$. 
\end{lemma}

\begin{remark}
	In the situation of the lemma above, the fibre $B^*/(B^*\cap\dot s B^*\dot s)$ has a natural group action coming from $\normalfont\text{Stab}_G(B^*)=B^*$. Since the automorphism group of $\mathbb A^1$ is the affine group $\mathbb G_a\rtimes \mathbb G_m$, there is a morphism  of group schemes $B^*\to \mathbb G_a\rtimes \mathbb G_m$. This morphism is in fact surjective. 
\end{remark}

The Zariski closure  $\overline{O(s)}$ of $O(s)$ in $X\times X$ admits a decomposition into the union of $O(s)$ and $O(e)$:
\[
\overline{O(s)}=\left\{ (B_0,B_1)\in X\times X\vert (B_0, B_1)\in O(s), \text{ or } (B_0, B_1)\in O(e)\right\}.
\]
Note that we have an identity section of the projection maps $\proj_i:\overline{O(s)}\to X, i=1,2,$ given by $X\cong O(e)\to \overline{O(s)}$. 
The complement of this section gives us $\proj_i: O(s)\to X$. We recall the following lemma.

\begin{lemma}\label{pr1triv}
	Let $s\in W$ be a simple reflection. The projection maps $\proj_i:\overline{O(s)}\to X, i=1,2,$ are a $\mathbb P^1$-bundles and they are locally trivial with respect to the Zariski open covering $\{\pi(\dot{w}U^+B^*)\}_{w\in W}$  of $X$.
	\end{lemma}
	\begin{proof}
		Let $P^*:=B^*sB^*\cup B^*$.  We have a cartesian diagram of $\Fpbar$-schemes:
	\begin{equation}\label{GP*}
\begin{tikzcd}
\overline{O(s)}\arrow[r, "\proj_2"]\arrow[d,"\proj_1"'] 
& G/B^* \arrow[d, "\pi_P"] \\
G/B^*\arrow[r, "\pi_P"]
& G/P^*,
\end{tikzcd}		
	\end{equation}
where  $\pi_P: G/B^*\to G/P^*$ is the projection map. All maps in this diagram are $G$-equivariant. Recall that $\pi_P$ is locally trivial with respect to the the (Zariski) open cover $\{\pi'(\dot w U^+B^*)\}_{w\in W}$ of $G/P^*$ \cite[\textsection II.1.10 (5)]{Jantzen}, where $\pi':G\to G/P^*$ is the canonical projection map. 
Since $\proj_i, i=1,2,$ is the base change of $\pi_P$, it  is a Zariski locally trivial fibre bundle. In particular, the fibres of $\proj_i$ are isomorphic to $P^*/B^*\cong \mathbb P^1$.
	\end{proof}

\subsection{The induced finite group action on the cohomology groups}\label{Gequiv}
For this section, let $\Gamma$ be a finite group and let $Y$ be a $\Fpbar$-scheme with $\Gamma$-action.  Following \cite[Definition 1.6]{MFK}, we explain how $\Gamma$ acts on the cohomology groups of $\Gamma$-equivariant $\Osh_Y$-modules.
\begin{dfn}
\normalfont
	Let $Y$ be a $\Fpbar$-scheme with $\Gamma$-action $\sigma:\Gamma\times Y\to Y$, and let $\mathcal V$ be an invertible sheaf of $\Osh_Y$-modules. Denote by $\mu:\Gamma\times \Gamma\to \Gamma$ the multiplication. A \emph{$\Gamma$-linearization} of  $\mathcal V$ constists of the datum of  an isomorphism of sheaves of  $\Osh_{\Gamma\times Y}$-modules,
	\[
	\phi: \sigma^*\mathcal V\overset{\sim}{\longrightarrow}\proj_2^*\mathcal V,
	\]
	such that $\left.\phi\right\vert_{\{1\}\times Y}$ is the identity and the cocycle condition on $\Gamma\times \Gamma\times Y$
	\[
	\left(\proj_{2,3}^*\phi\right)\circ\left( (\id_{\Gamma}\times \sigma)^*\phi\right)=(\mu\times \id_Y)^*\phi
	\]
	is satisfied.
\end{dfn}

We say that $\mathcal V$ is $\Gamma$-equivariant if it possesses a $\Gamma$-linearization.

\begin{example}
 (i)
  For $\mathcal V=\Osh_Y$, we naturally have $\sigma^*\Osh_Y\cong \Osh_{\Gamma\times Y}$ and $\proj_2^*\Osh_Y\cong\Osh_{\Gamma\times Y}$. Thus $\phi$ is given by the composition of these two isomorphisms. The cocycle condition follows because the pullback of the structure sheaf is the structure sheaf.  
Thus $\Osh_Y$ is $\Gamma$-equivariant.

 (ii) There is a natural morphism of sheaves of  $\Osh_{\Gamma\times Y}$-modules:
\[
\sigma^*\Omega^1_{Y}\overset{}{\longrightarrow}\Omega^1_{\Gamma\times Y}.
\]
Note that the projection map induces a projection map of sheaves of  $\Osh_{\Gamma\times Y}$-modules:
\[
\Omega^1_{\Gamma\times Y}\longrightarrow \proj_2^*\Omega^1_{Y}.
\]
The composition yields a morphism of $\Osh_{\Gamma\times Y}$-modules $\phi: \sigma^*\Omega^1_{Y}\to \proj_2^*\Omega^1_{Y}$. By checking on the level of stalks, one sees that $\phi$ is an isomorphism and the cocycle condition is satisfied. 

(iii) Let $r$ be an integer with $0<r\leq \dim_{\Fpbar}Y$. One shows that $\Omega_Y^{r}$ is $\Gamma$-equivariant by following the steps above and taking $r$-th exterior powers. Note that exterior powers commute with taking the inverse image. 

\end{example}

	We conclude this subsection with an observation that we will use later.  Suppose $H^0(Y,\Osh_Y)=\Fpbar$. Then $H^0(Y,\Osh_Y)$ is the trivial $\Gamma$-representation. Indeed, for $g\in \Gamma$, we have an isomorphism $\Osh_Y\overset{\sim}{\rightarrow}g_*\Osh_Y$.	Let $\varphi\in \Osh_Y (Y)$, for all $y\in Y$, 
\[
g.\varphi(y)=\varphi(g^{-1}.y).
\]
As we have $\Osh_Y(Y)=\Fpbar$, any $\varphi \in \Osh_Y(Y)$ is constant, and so $g.\varphi=\varphi$. Therefore $\Gamma$ acts on  $H^0(Y,\Osh_Y )$ trivially.

\subsection{Conventions for $G=\GL_n$}\label{GLdfn}

In the following, we will specifically consider the case  $G=\GL_n$. In this case,  we have $G^F={\GL_n}(\mathbb F_q)$.   
Let $T^*\subseteq \GL_n$ be the maximal torus such that $T^*(\Fpbar)$ corresponds to the diagonal matrices and let $B^*\subseteq \GL_n$ be the Borel subgroup such that $B^*(\Fpbar)$ corresponds to the upper triangular matrices.
Thus the $T^*\subseteq B^*$ we fixed are $F$-stable, so the Frobenius $F$ acts as the identity on $W$ cf.\cite[\textsection 4.2]{DM}.

\smallskip

There is a canonical isomorphism between the Weyl group $W$ associated to $T^*\subseteq B^*$ as above and the symmetric group $S_n$, in the sense that the elements of $W$ acts on $T^*(\Fpbar)$ by permuting the diagonal entries. We fix a set of generators $S:=\{s_1,\ldots,s_{n-1}\}$ of $W$ such that $s_i$ acts on the elements of $T^*(\Fpbar)$ by permuting the $i$-th and $(i+1)$-th entries.  Any product of the $s_i$'s in $W$ such that each $s_i$ shows up exactly once is called a Coxeter element. In particular, we call the product $s_1\cdots s_{n-1}$ the standard Coxeter element, which we denote by $\mathbf w$. Let $\{\alpha_1,\ldots,\alpha_{n-1}\}$ be the (positive) simple roots with each $\alpha_i$ corresponding to $s_i$. The the unipotent subgroup $U_{-\alpha_i}$ has $\Fpbar$-rational points consisting of matrices whose only nonzero entries lie on the diagonal and the $(i+1, i)$-th entry, with the entries on the diagonal all $1$'s. Moreover, each $U_{-\alpha_i}$ is isomorphic to the additive group $\mathbb G_a$. 
\smallskip

For any subset $I\subseteq S$, denote $W_I$ the subgroup of $W$ generated by $I$. We define the associated standard parabolic subgroup 
 cf. \cite[Theorem 29.2]{Humphreys}:
\[
P_I=B^*W_IB^*:=\bigcup_{w\in W_I}B^*\dot{w}B^*.
\]
Let $L_I$ be the standard Levi subgroup containing $T^*$ such that we have a Levi decomposition 
\[
P_I\overset{\sim}{\longrightarrow}U_I\rtimes L_I,
\]
where $U_I$ is the unipotent radical of $P_I$.

\subsection{Examples of Deligne--Lusztig varieties for $\GL_n$}\label{dhsandmore}

 Let $G=\GL_{n}$.  
 	Note that  the standard Coxeter element $\mathbf w$  corresponds to the $n$-cycle $(1,\ldots,n)$ in the symmetric group on $n$-elements. Recall from \cite[\textsection 2]{DL} that $ X(\mathbf w)$ can be identified with 
	the following subspace of the complete flag variety $X$:
	\[
	\left\{
	D_\bullet\Bigg\vert\dim_{\Fpbar} D_i=i,\ D_0=\{0\}, \ D_i=\bigoplus_{j=1}^{i} F^{j-1} D_1, \ i=1,\ldots,n-1, \ D_{n}=\Fpbar^{n}
	\right\}.
	\]
	Via the projection $D_\bullet\mapsto D_1$, one obtains an embedding of $X(\mathbf w)$ into $\mathbb P^{n-1}_{\Fpbar}$. In the notation of algebraic groups, this embedding is obtained from the projection $G/B^*\to G/P_I$, where  $I=\{s_2,\ldots, s_{n-1}\}$. Note that $G/P_I$ is isomorphic to $\mathbb P^{n-1}_{\Fpbar}$. In the former notation, one easily sees that there is an $\GL_n(\Fq)$-equivariant isomorphism
 \[
 X(\mathbf w)=X(s_1\cdots s_{n-1})\overset{\sim}{\longrightarrow} \mathbb P^{n-1}_{\Fpbar}-\mathcal H,
 \]
 where $\mathcal H$ is the union of all $\Fq$-rational hyperplanes in $\mathbb P^{n-1}_{\Fpbar}$. We denote the scheme on the right hand side by $\mathcal{X}^{n-1}_{\Fpbar}$.

  The smooth compactification $\overline{X}(\mathbf w)$ associated to the expression $\mathbf w=s_1\cdots s_{n-1}$ is isomorphic to the successive blow up $\widetilde{\mathcal{X}}^{n-1}_{\Fpbar}$ of $\mathbb P^{n-1}_{\Fpbar}$ along all $\Fq$-rational linear subschemes \cite[\textsection 4.1]{Ito} \cite[\textsection 4.1.2]{HWang}, cf. \cite[\textsection 2.5]{Linden}:
 \[
\widetilde{\mathcal{X}}^{n-1}_{\Fpbar}:=Y_{n-1}\longrightarrow Y_{n-2}\longrightarrow\cdots\longrightarrow Y_{-1}= \mathbb P^{n-1}_{\Fpbar},
 \]
where $Y_{i}\to Y_{i-1}$ is the blow up of $Y_{i-1}$ along the strict transform of all $\Fq$-rational linear subschemes $H\subseteq \mathbb P^{n-1}_{\Fpbar}$ with $\dim H=i$. The maps $Y_{i}\to Y_{i-1}$  are $\GL_n(\Fq)$-equivariant, so the map $\widetilde{\mathcal{X}}^{n-1}_{\Fpbar} \rightarrow \mathbb P^{n-1}_{\Fpbar}$ is equivariant under $\GL_n(\Fq)$-action. 

\medskip

Let $w\in W$ such that $w< \mathbf{w}$. Then $X_{\GL_n}(w)$ is isomorphic to a disjoint union of products of $\mathcal{X}^{j}_{\Fpbar}$ with $j< n-1$. This isomorphism extends to the corresponding smooth compactifications. We will discuss this example in more detail in Section \ref{levi}.

\section{The Weyl group and generalized Deligne--Lusztig varieties}\label{section2}
Assume $G$ to be split. 
The goal of this section is to recall some constructions related to the Weyl group $W$, and to give the definition of generalized Deligne--Lusztig varieties associated to an element of the free monoid $F^+$ (resp. $\hat F^+$). 
\subsection{Conjugacy classes and cyclic shifting} 

We start with reviewing some definitions and theorems from \cite[\textsection 3]{GP}.

\begin{dfn}\label{cyclicshiftdfn}
\normalfont
	We say $w,w'\in W$ are \emph{conjugate by cyclic shifts} when there exists a sequence of elements $v_0,\ldots,v_m\in W$ such that $v_0=w, v_m=w'$ and for all $i=1,\ldots,m$, we have $x_i, y_i\in W$ such that $v_{i-1}=x_iy_i$, $v_i=y_ix_i$, and $\ell(v_{i-1})=\ell(x_i)+\ell(y_i)=\ell(v_i) $. 
\end{dfn}

\begin{remark}
	Note that if $w, w'\in W$ are conjugate by cyclic shifts, then $\ell(w)=\ell(w')$. 
\end{remark}

The following theorem is from \cite[Theorem 3.1.4]{GP}.
\begin{theorem}[Geck--Pfeiffer]\label{gpcycl}
	Any two Coxeter elements of $W$ are conjugate by cyclic shifts. 
\end{theorem}

	Let $C$ be a conjugacy class of $W$. We write $C_{\min}$ for the subset of $C$ that consists of elements  with the shortest Bruhat length:
\[
C_{\min}:=\{v\in C \vert \ell(v)\leq \ell(w) \text{ for all }w\in C\}.
\]

\begin{dfn}\label{cyclicshiftdfn2}
\normalfont
	 Let $w,v\in W$. We write $w\rightarrow v$ if and only if there exists elements $w=w_0, w_1,w_2,\ldots,w_m=v\in W$, such that $w_{i}=t_iw_{i-1}t_i$ and $\ell(w_i)\leq \ell(w_{i-1})$ for $i=1,\ldots, m$, where $t_i\in S$. 
\end{dfn}

Now we may present a special case of the theorem from \cite[Theorem 3.2.9]{GP}, with wording adapted to our situation. 
\begin{theorem}[Geck-Pfeifer]\label{GP}
(i) Let $w\in W$, and let $C$ be a conjugacy class of $W$ containing $w$. Then there exists $w'\in C_{\min}$ such that $w\rightarrow w'$.
	\par\smallskip
	(ii) Let $w_1 , w_2\in W$ be two Coxeter elements, then 
	$w_1\rightarrow w_2$ and 
	$w_2\rightarrow w_1$. 
\end{theorem}

\subsection{Support}

\begin{dfn}\label{support}
\normalfont
	Let $w\in W$. The \emph{support} of $w$ is the following set:
	\[
	\normalfont\text{supp}(w):=\{s\in S\vert s\leq w\}.
	\]
\end{dfn}

Note that $\vert\normalfont\text{supp}(w)\vert \leq \ell(w)$ and that the equality holds when $w$ has a reduced expression $w=s_{i_1}\cdots s_{i_r}$ with $s_{i_j}\in S$ all distinct.

Let $G=\GL_{n}$ and $w\in W$. Set $I:=\normalfont\text{supp}(w)$, and let $C$ be the conjugacy class of $w$ in $W$. Let $W_I\subseteq W$ be the subgroup generated by $I$. Then there exists $w'\in C_{\normalfont\text{min}}$ such that $w\rightarrow w'$ and $w'$ is a Coxeter element in $W_I$. 
\subsection{The free monoid associated to the Weyl group}
Let us introduce the free monoid $F^+$, cf. \cite{DMR} and \cite[\textsection 2]{Orlik18}. 
\begin{dfn}
\normalfont
We define $F^+$ as the free monoid generated by the set of standard generators $S\subseteq W$.
\end{dfn}

\begin{remark}
	(i) There is a natural surjective morphism of monoids:
\[
\alpha: F^+\longrightarrow W
\]
with kernel generated by the relations in the group presentation of $W$.

(ii) There is a partial order $\preccurlyeq$ on $F^+$  defined by: $w'\preccurlyeq w$ whenever $w'=s_{i_1}\cdots s_{i_r}$ and $w=t_{1}\cdots t_{k},\  t_1,\ldots, t_k\in S$  such that $1\leq i_1\leq \cdots\leq i_r\leq k$, and $s_{i_j}=t_{i_j}$ for all $j=1,\ldots,r$.
 We call this the Bruhat order on $F^+$. Note that this is not entirely compatible with the Bruhat order on $W$.

(iii) There is a Bruhat length function on $F^+$, not compatible with the Bruhat length on $W$. When $w=s_{i_1}\cdots s_{i_r}\in F^+$, we have $\ell(w)=r$.

(iv) For $w, v\in F^+$, we always have  $\ell(wv)=\ell(w)+\ell(v)$. 
\end{remark}

There is also a variant of $F^+$ defined in \cite[p. 22]{Orlik18}.
\begin{dfn} \label{hatF}
\normalfont
Let $\widehat W$ be a copy of $W$. 
	Define $\hat{F}^+$ as the free monoid generated by $S\dot\cup T'$, where
	\[
	T':=\left\{\left.\widehat{sts}\in \widehat{W}\right\vert st\neq ts\normalfont\text{ in } W,  s,t\in S\right\}.
	\]
\end{dfn}

\begin{remark}
Note that $T'$ and $W$ are forced to be disjoint in $\hat{F}^+$.
\end{remark}

	We define the Bruhat length function on $\hat F^+$ as the function counting the number of elements of $S$ and $\widehat S$ showing up in the expression, where $\widehat S$ is the set of generators in $\widehat{W}$ corresponding to $S$, instead of counting the number of generators.

\subsection{Constructing generalized Deligne--Lusztig varieties}
As in Section \ref{dfn}, we may define generalized Deligne--Lusztig varieties for elements of the free monoids $ F^+$ and $ \hat{F}^+$.
They are introduced in \cite[Example 3.10 (c)]{Lusztig78}, and one may refer to \cite[\textsection 2.2.11]{DMR} and \cite[end of \textsection 3]{Orlik18} for more discussions. We only recall the definitions and properties here. 

 \begin{dfn}\label{dfndlv}
 \normalfont
 For $s\in F^+$ with $\alpha(s)\in S$. We set $O(s):=O(\alpha(s))$. For each   $w=s_{i_1}\cdots s_{i_r}\in F^+$, 
 define
 	 		\[
	O(s_{i_1},\ldots,s_{i_r}):=\left\{(B_0,\ldots,B_r)\in X^{r+1}\vert (B_{j-1},B_j)\in O(s_{i_j}), j=1,\ldots,r \right\},
	\]
	and the corresponding \emph{Deligne--Lusztig variety}
		\[
	X(w):=O(s_{i_1},\ldots,s_{i_r})\times_{\left(X\times X\right)}\Gamma_F,
	\]
	where $O(s_{i_1},\ldots,s_{i_r})\to X\times X$ is the projection map $(B_0,\ldots,B_r)\mapsto (B_0, B_r)$. We may alternatively write
 	 	\begin{multline*}
 	 		X(w):=\\
 	 		\Large\{(B_0,\ldots,B_r)\in X^{r+1}\vert(B_{j-1},B_j)\in O(s_{i_j}), \forall j=1,\ldots,r, B_r=F(B_0)
 			\Large\}.
 	 	\end{multline*}
 \end{dfn}
 
 \begin{remark}\label{isoofgeneralizeddlvar}
The scheme $X(w)$ is a quasi-projective smooth $\Fpbar$-scheme of dimension $r$ with a left $G^F$-action. 
 	For any $t_1,\ldots,t_k\in W$, the scheme  $O(t_1,\ldots,t_k)$ is defined analogously as in Definition \ref{dfndlv}. 
 
 	 Let $w=s_{i_1}\cdots s_{i_r}\in F^+$. If $\ell(w)=\ell(\alpha(w))$, then there is a $G^F$-equivariant isomorphism 
 	 \[
 	 \begin{array}{rcl}
 	 	 X(w) & \overset{\sim}{\longrightarrow} & X(\alpha(w))\\
 	 	 (B_0,\ldots,B_r) & \longmapsto & B_0
 	 \end{array}
 	 \]
 	 of $\Fpbar$-schemes \cite[Example 3.10 (c)]{Lusztig78} cf. \cite[p. 759]{DMR}. Note that if we consider $w\in W$, then for each reduced expression $s_{i_1}\cdots s_{i_r}$ of $w$ with $s_{i_j}\in S$, we get an element of $F^+$ and a corresponding generalized Deligne--Lusztig variety. 
 \end{remark}

Similarly, we define the generalized Deligne--Lusztig variety corresponding to elements in $\hat{F}^+$. After post composing with the isomorphism $W\cong \hat W$, we may extend the surjective map $\alpha$ to $\hat{\alpha}: \hat{F}^+ \to W$. 

 \begin{dfn}
 \normalfont
 For $\widehat{sts}\in \hat{F}^+$, we set $O(\widehat{sts}):=O(\hat{\alpha}(\widehat{sts}))$. For each $w=t_{1}\cdots t_{r}\in \hat F^+$, define
 the corresponding \emph{Deligne--Lusztig variety}
		\[
	X(w):=O(t_1,\ldots,t_r)\times_{\left(X\times X\right)}\Gamma_F.
	\]
	 We may alternatively write
 	 	\begin{multline*}
 	 		X(w):=\\
 	 		\left\{\left.(B_0,\ldots,B_r)\in X^{r+1}\right\vert (B_{j-1},B_j)\in O(t_{j}), \forall j=1,\ldots,r, B_r=F(B_0)\right\}.
 	 	\end{multline*}
 \end{dfn}
 \begin{remark}
This is a quasi--projective smooth $\Fpbar$-scheme with a left $G^F$-action. 
		Since $O(a)\times_X O(b)\overset{\sim}{\rightarrow}O(a,b)$ for all $a, b\in \hat{F}^+$,
		we know that $X(w)$ has dimension $\ell(w)\geq r$.
		 \end{remark}

 \subsection{Smooth compactification of generalized Deligne--Lusztig varieties}\label{generalsmthcompt}
  We write down the smooth compactifications for Deligne--Lusztig varieties corresponding to $w\in F^+$ and $w\in \hat{F}^+$. They are the same as the construction in \cite[\textsection 9]{DL} when $w\in F^+$ or $w\in \hat{F}^+$ is the reduced expression for some $w'\in W$. 
 \begin{dfn}
 \normalfont
 	 	Let  $w=s_{i_1}\cdots s_{i_r}\in F^+$. Define the following $\Fpbar$-scheme:
 	 	\begin{multline*}
 	 		\overline{X}(w):=\\
 	 		\left\{\left.(B_0,\ldots,B_r)\in X^{r+1}\right\vert (B_{j-1},B_j)\in \overline{O(s_{i_j})}, \forall j=1,\ldots,r, B_r=F(B_0)\right\}.
 	 	\end{multline*}
 \end{dfn}

\begin{remark}
	The $\Fpbar$-scheme $\overline{X}(w)$ is smooth projective with a left $G^F$-action cf. \cite[\textsection 9]{DL}, \cite[\textsection 2.3.1]{DMR}. It is a \emph{smooth compactification} of $X(w)$.
	
 	Let $w=s_{i_1}\cdots s_{i_r}$ be an element in $ F^+$ such that it corresponds to a reduced expression of $\alpha(w)\in W$. Then $\overline{X}(w)$ is the same as $\overline{X}(\alpha(w))$, corresponding to the reduced expression $s_{i_1}\cdots s_{i_r}$, constructed in Definition \ref{smoothcompactifi}. 
\end{remark}

Similarly, we have the following definition for $w\in \hat{F}^+$.
\begin{dfn}\label{XoF+}
\normalfont
 For $\widehat{sts}\in \hat{F}^+$, we set $\overline{O(\widehat{sts})}:=\overline{O(\hat{\alpha}(\widehat{sts}))}$. 
 	 Let  $w=t_{1}\cdots t_{r}\in \hat F^+$. We define the following $\Fpbar$-scheme containing $X(w)$:
 	 \begin{multline*}
 	 	\overline{X}(w):=\\
 	 	\left\{\left.(B_0,\ldots,B_r)\in X^{r+1}\right\vert (B_{j-1},B_j)\in \overline{O(t_{j})}, \forall j=1,\ldots,r, B_r=F(B_0)\right\}.
 	 \end{multline*}
  \end{dfn}
  
  \begin{remark}
  As before,	the $\Fpbar$-scheme $\overline{X}(w)$ has a left $G^F$-action. 
	 	\end{remark}

\begin{lemma}\label{XoF+smooth}
Let $w=t_{1}\cdots t_{r}\in \hat F^+$, then $\overline{X}(w)$ is projective and smooth. 
\end{lemma}

\begin{proof}
Since  $\overline{O(t_j)}$ is projective for all $j=1,\ldots,r$, it follows from \cite[Proposition 2.3.6 (iv)]{DMR} that  $\overline{X}(w)$ is projective. 

Note that $\overline{O(t)}$ are smooth for all $t\in S$. 
	We may have $t_{j}, j=1,\ldots,r$ to be of the form $t_{j}=\widehat{sts}$, where $st\neq ts, s,t,\in W$. Since $st\neq ts$, we see that $s$ and $t$ do not correspond to non-adjacent simple reflections in $S_n$. Now let $I=\{s,t\}$, the parabolic subgroup $W_I$ of the Weyl group $W$ is thus isomorphic to the symmetric group $S_3$. In particular, $sts$ is a reduced expression of the longest element in $W_I$. Observe that the projection maps $\overline{O(sts)}\to X$ are both Zariski locally trivial fibre bundles with fibres isomorphic to the flag variety of $\GL_3$. Thus $\overline{O(sts)}$ is a smooth $\Fpbar$-scheme cf. \cite[Corollary 2.2.10]{DMR}.
	
	Hence for all $j=1,\ldots,r$, $\overline{O(t_j)}$ is smooth. By \cite[Proposition 2.3.5]{DMR}, we may conclude that $\overline{X}(w)$ is smooth. 
\end{proof}

 \subsection{Stratifications of $\overline{X(w)}$ and $\overline{X}(w)$}
 Let $w_1,w_2\in W$ such that $\ell(w_1w_2)=\ell(w_1)+\ell(w_2)$, recall that we have an isomorphism of schemes $O(w_1)\times_XO(w_2)\overset{\sim}{\longrightarrow} O(w_1w_2)$. Hence for $w,w'\in W$,
 $w'\leq w$ if and only if $\overline{O(w')}\subseteq \overline{O(w)}$. This implies that we have a stratification as follows for any $w\in W$, cf. \cite[\textsection 1.2]{DL}, \cite[\textsection II.13.7]{Jantzen},
 \begin{equation}\label{strat}
 	\overline{O(w)}=\bigcup_{w'\leq w}O(w').
 \end{equation} 
 The intersection of the graph of the Frobenius $\Gamma_F$ and any of strata $O(w')$ above is transversal. For any $w\in W$, the intersection of (\ref{strat}) and $\Gamma_F$ yields a similar stratification of $\overline{X(w)}$,
  \[
 \overline{X(w)}=\bigcup_{w'\leq w} X(w').
 \] 
 Let $w\in F^+$ or $ \hat{F}^+$, with the expression $s_{i_1}\cdots s_{i_r}$, then for all $w'\preceq w$ a subword, we know that $X(w')$ is isomorphic to a locally closed subscheme of $\Xo(w)$, thus we have the following stratification, cf. \cite[\textsection 3]{Orlik18}, 
 \[
 \Xo(w)=\bigcup_{w'\preceq w}X(w').
 \]

\begin{example}
Let $G={\GL_n}$. Let $m\leq n-1$ be a positive integer and $w\in W$ with a reduced expression $s_{i_1}\cdots s_{i_{m}}$ such that all $s_{i_j}$, $j=1,\ldots,m,$ are distinct. 
 Recall from  \cite[Lemma 9.11]{DL} that 
 \[
 D=\bigcup_{w'\prec w}X(w'),
 \]
 with each $w'\prec w$ considered as a subword of $s_{i_1}\cdots s_{i_{m}}\in F^+$, is the normal crossing divisor of $\Xo(w)$ at infinity. 

 Note that the projection map
\[\arraycolsep=1pt\def\arraystretch{1.5}
		\begin{array}{rcl}
			 X^m &\longrightarrow &  X\\
			(B_0,\ldots,B_{m-1})& \longmapsto &  B_0
		\end{array}
		\]
 induces an isomorphism on the open subschemes
 \[
 \Xo(w)\backslash D\longrightarrow X(w).
 \]
 This map extends to the Zariski closure of $X(w)$ in $X$, and so we have a surjective morphism
 \[
 \Xo(w)\longrightarrow \overline{X(w)}.
 \]
 Since $w=s_{i_1}\cdots s_{i_{m}}$ with all $s_{i_j}$ distinct,  $\Xo(w)$ and $\overline{X(w)}$ have stratifications indexed by the same set, and each corresponding strata is isomorphic.  
 In fact, they are isomorphic as $\Fpbar$-schemes \cite[Lemma 1.9]{Hansen}. 	
\end{example}

\section{Geometry of Deligne--Lusztig varieties via $\mathbb P^1$-bundles}\label{sws}
Assume $G$ to be split. 
In this section, we consider $\mathbb P^1$-bunldes $\pi_1:\Xo(sws)\to \Xo(ws) $ and $\pi_2:\Xo(sws)\to \Xo(sw)$ constructed from the morphism $\overline{O(s)}\to X$ from Section \ref{fibrations}.

\subsection{The structure of certain morphisms as $\mathbb P^1$-bundles}\label{P1}
Let $w=t_{1}\cdots t_{r}\in \hat F^+$ and $s\in S$. We fix the notations for the smooth  compactifications of the Deligne--Lusztig varieties $X(sws)$, $X(ws)$:
		\begin{multline*}
			\Xo (sws)=\\
			\left\{(B_0,\ldots,B_{r+2})\in X^{r+3}\left\vert 
		\begin{array}{l}
			(B_{j},B_{{j+1}})\in \overline{O(t_{j})}, j=1,\ldots,r, \\
			(B_0,B_1)\in \overline{O(s)}, 	(B_{r+1}, B_{r+2})\in \overline{O(s)}, B_{r+2}=FB_0
		\end{array}
		\right.\right\},
		\end{multline*}
		
		\begin{equation*}
			\Xo (ws)=\left\{(B'_0,\ldots,B'_{r+1})\in X^{r+2}\left\vert 
		\begin{array}{l}
			(B'_{j-1},B'_j)\in \overline{O(t_{j})}, j=1,\ldots,r, \\
			(B'_{r},B'_{r+1})\in \overline{O(s)}, 	B'_{r+1}=FB'_0
		\end{array}
		\right.\right\}.
		\end{equation*}

		\begin{lemma}\label{wsP1}
		 The map $\pi_1:\Xo(sws)\to \Xo(ws)$ defined by 
		\[
		\left(B_0,B_1,\ldots,B_{r+1}, FB_0\right)\mapsto\left (B_1,B_2,\ldots,B_{r+1}, FB_{1}\right)
		\]
		is a $\mathbb P^1$-bundle over $\Xo(ws)$ locally trivial with respect to a Zariski covering of $\Xo(ws)$.
				Furthermore, $\pi_1$ has a  section $\sigma:  \Xo(ws)\to \Xo(sws)$ defined by 
		\[
		 \left(B'_0,\ldots,B_r', FB'_{0}\right)\mapsto \left(B_0',B'_0,\ldots,B_r', FB'_{0}\right)
		\]
		with $\pi_1\circ\sigma={\normalfont\text{id}}_{\Xo(ws)}$.
		\end{lemma}
		\begin{proof}
		Let us first check the well-definedness of $\pi_1$. Take $(B_0,B_1,\ldots,B_{r+1}, FB_0)\in \Xo(sws)$. Since $(B_0,B_1)\in \overline{O(s)}$, and $F$ fixes any $s\in S$, we know that $(FB_0, FB_1)\in \overline{O(s)}$. As we also have $( B_{r+1}, FB_0)\in \overline{O(s)}$, 
		by \cite[\textsection 1.2 (b1)]{DL}, 
		we have $(B_{r+1}, FB_1)\in \overline{O(s)}$.  Thus $(B_1,B_2,\ldots,B_{r+1}, FB_{1})\in \Xo(ws)$.

		To see that $\pi_1$ gives a $\mathbb P^1$-bundle, take any $(B'_0,\ldots,B'_{r},FB'_{0})\in \Xo(ws)$, and take any $B_0\in X$ such that $(B_0, B_0')\in \overline{O(s)}$. Since $F$ fixes $s$, we have $(FB_0, FB_0')\in\overline{O(s)}$. As $(B'_{r},FB_0')\in \overline{O(s)}$, we know by \cite[\textsection 1.2 (b1)]{DL} that 
		$(FB_0, B'_r)\in \overline{O(s)}$. Thus $(B_0,B'_0,\ldots,B'_{r}, FB_0 )\in \Xo(sws)$, and the preimage of $(B'_0,\ldots,B'_{r},FB'_{0})$ under $\pi_1$ is 
		\[
		 \left\{\left(B_0,B'_0,\ldots,B'_{r}, FB_0 \right)\in \Xo(sws)\left\vert (B_0, B_0')\in \overline{O(s)}\right.\right\}.
		\]
		Thus the fibre of $\pi_1$ at any $(B'_0,\ldots,B'_{r},FB'_{0})$ is isomorphic to the fibre of $\overline{O(s)}\to X$ at $B_0'$. 
		
		Let $\proj_{2,\ldots,r+3}: \overline{O(s)}\times X^{r+1}\to X^{r+2}$ be the projection map to the 2nd to $(r+3)$-th  component.
		If we take the embedding of $\Xo(ws)$ into $X^{r+2}$, we find that $\Xo(sws)$ is isomorphic to the fibre product of $\Xo(ws)\hookrightarrow X^{r+2}$ and $\proj_{2,\ldots,r+3}$. 
	
	 	\begin{equation*}
\begin{tikzcd}
\Xo(sws) \arrow[r]\arrow[d,"\pi_1"'] 
& \overline{O(s)}\times X^{r+1} \arrow[d, "\proj_{2,\ldots,r+3}"] \\
\Xo(ws)\arrow[r, hook]
& X^{r+2}.
\end{tikzcd}		
	\end{equation*}
  Let $\pi_G:G\to G/B^*$ be the canonical projection map for $X=G/B^*$. By Lemma \ref{pr1triv}, we know that $\proj_2: \overline{O(s)}\to X$ is locally trivial with respect to the Zariski covering $\{\pi_G(\dot{w}U^+B^*)\}_{w\in W}$  of $X$. Thus $\proj_{2,\ldots,r+3}$ is locally trivial with respect to the Zariski covering  $\{\pi_G(\dot{w}U^+B^*)\times X^{r+1}\}_{w\in W}$ of $X^{r+2}$.
Via embedding $\Xo(ws)$ into $X^{r+2}$, we see that $\pi_1$ is locally trivial with respect to a Zariski open covering of $X(ws)$.
		Therefore $\Xo(sws)$ is a $\mathbb P^1$-bundle over $\Xo(ws)$. 
		
		Finally, the  statement $\pi_1\circ\sigma={\normalfont\text{id}}_{\Xo(ws)}$ can be easily verified.
					\end{proof}

We also have a smooth compactification of the Deligne--Lusztig variety $X(sw)$:
\begin{equation*}
	\Xo (sw)=\left\{(B'_0,\ldots,B'_{r+1})\in X^{r+2}\left\vert 
		\begin{array}{l}
			(B'_{j},B'_{j+1})\in \overline{O(t_{j})}, j=1,\ldots,r, \\
			(B'_0,B'_1)\in \overline{O(s)}, 	B'_{r+1}=FB'_0
		\end{array}
		\right.\right\}.
\end{equation*}
		\begin{lemma}\label{swP1}
		  The map $\pi_2:\Xo(sws)\to \Xo(sw)$ defined by 
		\[
		\left(B_0,B_1,\ldots,B_{r+1}, FB_0\right)\mapsto\left (B_{r+1}, FB_1,\ldots,FB_{r+1}\right)
		\]
		is a $\mathbb P^1$-bundle over $\Xo(sw)$ locally trivial with respect to an fppf-covering of $\Xo(sw)$.
				\end{lemma}
		\begin{proof}
	We first check the well-definedness of $\pi_2$. Take $(B_0,B_1,\ldots,B_{r+1}, FB_0)\in \Xo(sws)$. Since $(B_{r+1}, FB_0)\in \overline{O(s)}$ and $(FB_0, FB_1)\in \overline{O(s)}$, we know that $(B_{r+1}, FB_1)\in\overline{O(s)}$.  Thus $(B_{r+1}, FB_1,\ldots,FB_{r+1})\in \Xo(sw)$.

Note that $F:\Xo(sw)\to \Xo(sw)$ is a flat and finite morphism. Via flat base change by $F:\Xo(sw)\to \Xo(sw)$, $\Xo(sws)$ becomes a fppf $\mathbb P^1$-bundle over $\Xo(sw)$. More precisely, we have a cartesian square:

\begin{equation}\label{Fbasechange}
\begin{tikzcd}
Y \arrow[r, "\proj_1'"]\arrow[d,"\pi_2'"'] 
& \Xo(sws) \arrow[d, "\pi_2"] \\
\Xo(sw)\arrow[r, "F"]
& \Xo(sw),
\end{tikzcd}
\end{equation}

where  
\[
 Y:=\{(B,B')\in \Xo(sws)\times\Xo(sw)\vert FB'=\pi_2(B)\},
\]
and $\pi_2'$ is projection to the second component cf. \cite[Theorem 1.6]{DL}. Moreover, $\pi_2'$ fits into another cartesian square:
\begin{equation*}
\begin{tikzcd}
Y \arrow[r, "\iota"]\arrow[d,"\pi_2'"'] 
& \overline{O(s)} \arrow[d, "\proj_2"] \\
\Xo(sw)\arrow[r,"\proj_0"]
& X.
\end{tikzcd}
\end{equation*}
Let $B=(B_0,\ldots,B_{r+1},FB_0)$ and $B'=(B_0',B_1',\ldots,B_r', FB_0')$ such that $(B, B')\in Y$. The map $\proj_0$ is given by 
\[
(B_0',B_1',\ldots,B_r', FB_0')\mapsto B_0'.
\]
Also note that $\iota$ is the map defined by 
\[
\left((B_0,\ldots,B_{r+1},FB_0),(B_0',B_1',\ldots,B_r', FB_0')\right)\mapsto (B_0,B_0').
\]
Since $F(s)=s$, we know that $(B_0, B_0')$ and $(FB_0, FB_0')$ must belong to the same orbit in $X\times X$. The condition $FB_0'=B_{r+1}$ implies that $(FB_0, FB_0')\in \overline{O(s)}$, so $(B_0, B_0')\in \overline{O(s)}$. 

For all $B'=(B_0',B_1',\ldots,B_r', FB_0')\in \Xo(sw)$, we have by the condition $FB'=\pi_2(B)$ that $\pi_2'^{-1}(B')=(B, B')$ where $B=(B_0,\ldots,B_{r+1}, FB_0)$ such that $FB_0'=B_{r+1}$. In particular, $(B_0, B_1')\in \overline{O(s)}$. Conversely, for any $B_0\in X$ such that $(B_0, B_0')\in \overline{O(s)}$, 
we have $(FB_0,FB_0')\in\overline{O(s)}$ and $(B_0,B_1')\in \overline{O(s)}$. 
Thus $(B_0, B_1',\ldots,$ $B_r',FB_0', FB_0)\in \Xo(sws)$. In particular, 
\[
\left( (B_0, B_1',\ldots,B_r',FB_0', FB_0),(B_0',B_1',\ldots,B_r', FB_0')\right)\in \pi_2'^{-1}(B_0',B_1',\ldots,B_r', FB_0').
\]

Hence via this cartesian square, we know that $\pi_2'$ is a $\mathbb P^1$-bundle over $\Xo(sw)$. The argument that $\pi_2'$ is locally trivial with respect to the Zariski topology is analogous to the one used in Lemma \ref{wsP1}. 

Finally, we return to the cartesian diagram (\ref{Fbasechange}). Since $\Xo(sw)$ is a $\Fpbar$-scheme of finite type, we know that the Frobenius endomorphism $F:\Xo(sw)\to \Xo(sw)$ is  flat, and is a universal homeomorphism. Thus the base change $\proj_1'$ of $F$ is also flat, and is a universal homeomorphism. Since $\pi_2'$ gives a Zariski locally trivial $\mathbb P^1$-bundle, and the fppf topology is finer than the Zariski topology, there exists a fppf open covering $\mathcal U:=\{f_i:U_i\to \Xo(sw)\}_i$ such that $\pi_2'$ is locally trivial with respect to $\pi_2'$. We get a  composition of cartesian diagrams for each $i$:

\begin{equation*}
	\begin{tikzcd}
U_i\times \mathbb P^1 \arrow[r] \arrow[d] & Y \arrow[r, "\proj_1'"]\arrow[d,"\pi_2'"'] 
& \Xo(sws) \arrow[d, "\pi_2"] \\
U_i\arrow[r, "f_i"] & \Xo(sw)\arrow[r, "F"]
& \Xo(sw).
\end{tikzcd}
\end{equation*}
Since $F$ is flat and surjective, we see that the composition morphism $F\circ f_i$ is flat, locally of finite presentation for all $i$ and 
\[
\bigcup_i \left(F\circ f_i(U_i)\right)=\Xo(sws).
\]
Thus $\mathcal U':=\{F\circ f_i: U_i\to \Xo(sw)\}$ gives a fppf covering for $\Xo(sw)$ such that $\pi_2$ is locally trivial with respect to $\mathcal U'$. 
		\end{proof}

		\begin{remark}
			Recall that proper morphisms are preserved under fpqc (hence fppf) base change and descent \cite[Expos\'{e} VIII, Corollary 4.8]{SGA1}. We see that the map $\pi_1$ and $\pi_2$ in Lemma \ref{wsP1} and \ref{swP1} are also proper morphisms of $\Fpbar$-schemes. 
			
			Note that the constructions of $\pi_1$ and $\pi_2$ in
			 Lemma \ref{P1} and \ref{swP1} use the fact that $F(s)=s$. 
		\end{remark}	 	
		
\begin{lemma}
 The maps  $\pi_1:\Xo(sws)\to \Xo(ws)$ and $\pi_2:\Xo(sws)\to \Xo(sw)$  defined above are $G(\Fq)$-equivariant. 
\end{lemma}
\begin{proof}
 Take $(B_0,B_1,\ldots,B_{r+1}, FB_0)\in \Xo(sws)$. Recall that $G(\Fq)$ acts on $\Xo(sws)$ via conjugation in each component. Let $g\in G(\Fq)$. 
		\[
		\pi_1\left(gB_0g^{-1},\ldots,gB_{r+1}g^{-1}, g(FB_0)g^{-1}\right)=\left(gB_1g^{-1},\ldots,gB_{r+1}g^{-1}, F(gB_1g^{-1})\right)
		\]
		Since any $g\in G(\Fq)$ is fixed by $F$, we have $F(gB_1g^{-1})=g(FB_1)g^{-1}$ and thus
		\[
		 \pi_1\left(gB_0g^{-1},\ldots,gB_{r+1}g^{-1}, g(FB_0)g^{-1}\right)=g.\pi_1\left(B_0,B_1,\ldots,B_{r+1}, FB_0\right).
		\]

	For $\pi_2$, let $(B_0,B_1,\ldots,B_{r+1}, FB_0)\in \Xo(sws)$, and $g\in G(\Fq)$. Consider the following:
	\begin{multline*}
		\pi_2\left(gB_0g^{-1},gB_1g^{-1},\ldots, g(FB_0)g^{-1}\right)= \\\left(gB_{r+1}g^{-1},F(gB_1g^{-1}),\ldots,F(gB_{r+1}g^{-1})\right).
	\end{multline*}
Since $F$ fixes $g$ and $g^{-1}$, we have $F(g^{-1}B_i g)=B_i$ for all $i$.  Thus
\[
 \pi_2\left(gB_0g^{-1},gB_1g^{-1},\ldots, g(FB_0)g^{-1}\right)=g.\pi_2 \left(B_0,B_1,\ldots,B_{r+1}, FB_0\right).
\]
\end{proof}

\subsection{Cohomology of the structure sheaf of the $\mathbb P^1$-bundles}
			 Recall that the smooth compactifications of Deligne--Lusztig varieties are smooth, separated schemes of finite type over $\Fpbar$. 
		
			\begin{proposition}\label{p1coh}
		 Let $w\in \hat F^+$, $ s\in S$. Then for all $k\geq 0$, there are $G(\Fq)$-equivariant isomorphisms of $\Fpbar$-vector spaces:
		 \[H^k\Big(\Xo(ws),\mathcal O_{\Xo(ws)}\Big)\overset{\sim}{\longrightarrow} H^k\Big(\Xo(sws),\mathcal O_{\Xo(sws)}\Big)\]
		 and
		 \[H^k\Big(\Xo(sw),\mathcal O_{\Xo(sw)}\Big)\overset{\sim}{\longrightarrow} H^k\Big(\Xo(sws),\mathcal O_{\Xo(sws)}\Big).
		 \]
		\end{proposition}
	
	\begin{proof}
	To simplify our notations, we use $X:=\Xo(sws), Y:= \Xo(ws)$ (resp. $Y:=\Xo(sw)$). Let $\pi$ be $\pi_1$ as in Lemma \ref{wsP1} (resp. be $\pi_2$ as in Lemma \ref{swP1}), and 
		consider the Leray spectral sequence for $\mathcal O_X$:
		\begin{equation}\label{YXspectral}
			E_2^{i,j}=H^i\left(Y,R^j\pi_*\mathcal O_X\right)\implies H^{i+j}\left(X,\mathcal O_X\rightp.
		\end{equation}
		Note that  $X, Y$ are smooth schemes of finite type over $\Fpbar$. Since the (closed) fibres of $\pi$ are equidimentional and $\dim_{\Fpbar} X=\dim_{\Fpbar}Y+\dim_{\Fpbar}\P^1_{\Fpbar}$, by miracle flatness \cite[Theorem 23.1]{Matsumura}, we know that $\pi:X\to Y$ is a flat morphism. Since the fibres are isomorphic to $\mathbb P^1_{\Fpbar}$, they are geometrically connected and geometrically reduced. By \cite[\textsection 9.3.11]{FGAexplained}, we know that  $\pi_*\mathcal O_X\cong \mathcal O_Y$.

        Note that $X$ is geometrically reduced and $\pi$ has connected fibres. By \cite[Prop 1.4.10]{EGA31}, we know that for $j>0$, $R^j\pi_*(\mathcal O_X)$ are coherent $\Osh_Y$-modules. We shall show that $R^j\pi_*\mathcal{O}_X=0$ for all $j>0$. Note that for all $y\in Y$, we have $X\times_Y y\cong \mathbb{P}^1$, so $H^j(X\times_Y y, \Osh_{X\times_Y y})=0$ for all $j>0$. Since in addition $\pi$ is proper and flat, it follows from \cite[\textsection 25.1.5, \textsection 25.1.6]{Vakil} that $\psi_y^*R^j\pi_*(\mathcal O_X)=0$ for all $j>0$, where $\psi_y:y\to Y$ is the inclusion map. This implies that $R^j\pi_*(\mathcal O_X)$ has trivial stalks for all $y\in Y$ and $j>0$. Hence $R^j\pi_*(\mathcal O_X)=0$ for all $j>0$.

			  	Therefore the Leray spectral sequence (\ref{YXspectral}) degenerates and we have 
		\[
		H^i(Y,\mathcal O_Y)\overset{\sim}{\longrightarrow} H^i(X,\mathcal O_X)
		\]		
		for all $i\geq 0$. Since $\pi$ is a $G(\Fq)$-equivariant morphism, these isomorphisms are $G(\Fq)$-equivariant. 		
	\end{proof}

\section{Towards induction steps}\label{steps}

Let $G=\GL_n$ as in \textsection \ref{GLdfn}.
 We now set the stage for the (double) induction. Our goal is to reduce the problem of computing the cohomology groups of coherent sheaves on the smooth compactifications of Deligne--Lusztig varieties to those corresponding to a Coxeter element of $W$ or a Coxeter element corresponding to a standard parabolic subgroup $P\subseteq \GL_{n}$.

 In loose terms, we may describe the strategy as follows: 
 Let $w$ be an element of the free monoid $F^+$ or $\hat F^+$. Its expression may contain a repeating $s\in S$. 
 We introduce operations $C, K, R$ on $F^+$ and $\hat F^+$ so that after applying finitely many such operations on $w$, we may obtain a word of the form $sw's$. The operations $C,K,R$ preserve the length of $w$, so we will still have $\ell(w)=\ell(sw's)=\ell(w')+2$. Then Section \ref{sws} helps us to reduce this to the case of $sw'$ and thus removing one of the repeating $s$. This process has finitely many steps and we will eventually reduce it to the case of $v\in W$ being a product of non-repeating simple reflections with $\ell(v)=\vert\normalfont\text{supp}(w)\vert $. 
 
 We will introduce each of the operations $C,K, R$ and discuss how they affect the cohomology groups of the corresponding smooth compactifications of Deligne--Lusztig varieties. 
 
One can find the original definitions of these operations 
and the double induction strategy in \cite[before Proposition 7.9]{Orlik18} for the case of  $\ell$-adic cohomology with compact support.

\subsection{The Cyclic shifting operation}
The elements $sw', w's\in W$ are conjugated by $s\in S$. Recall from Definition \ref{cyclicshiftdfn} and \ref{cyclicshiftdfn2}
that this can be generalized to elements of $W$ being conjugate by cyclic shift. The following operator is constructed to impose the concept of elements being conjugate by cyclic shift on $F^+$ and $\hat F^+$.

\begin{dfn}\label{operatorC}\normalfont
	Let $w\in F^+$ \big(resp. $\hat{F}^+$\big). If $w=sw'$, where $w'\in F^+$ \big(resp. $\hat{F}^+$\big) and $s\in S$, we define the operator $C$ on $F^+$ \big(resp. $\hat{F}^+$\big) by $C(w):=w's$. 
	\end{dfn}

	\begin{proposition}\label{cyclicshift}
	Let $w\in \hat{F}^+$, such that $w=sw'$ with $s\in S$. Then we have isomorphisms of $\Fpbar$-vector spaces for all $k\geq 0$:
	\[
	H^k\left(\overline{X}(w), \Osh_{\overline{X}(w)}\right)\overset{\sim}{\longrightarrow} H^k \left(\overline{X}(C(w)), \Osh_{\overline{X}(C(w))}\right).
	\]
	Furthermore, the isomorphism are $\GL_n(\Fq)$-equivariant.
\end{proposition} 
	\begin{proof}
Let $w\in \hat{F}^+$, such that $w=sw'$ with $s\in S$ and $ w'\in \hat{F}^+$. 

Consider the product $ws=sw's$ in $\hat{F}^+$.	We consider from Section \ref{P1} the surjective morphisms $\pi_1: \Xo(sw's)\to\Xo(w's)$ and $\pi_2:\Xo(sw's)\to \Xo(sw')$ that make $\Xo(sw's)$ a fppf $\mathbb P^1$-bundle over $\Xo(w's)$ and $\Xo(sw')$ respectively. 
	
 By Proposition \ref{p1coh}, we have $\GL_n(\Fq)$-equivariant isomorphisms for all $k\geq 0$,

\[
H^k\leftp\Xo(sw'),\mathcal O_{\Xo(sw')}\rightp\cong H^k\leftp\Xo(sw's),\mathcal O_{\Xo(sw's)}\rightp\cong H^k\leftp \Xo(w's),\mathcal O_{\Xo(w's)}\rightp.
\]
Thus 
\[
H^k\leftp\Xo(w),\mathcal O_{\Xo(w)}\rightp\cong H^k\leftp \Xo(C(w)),\mathcal O_{\Xo(C(w))}\rightp,
\]
and this concludes the proof. 
	\end{proof}

\subsection{Operations corresponding to relations}

Recall from Definition \ref{hatF} that $\hat{F}^+$ is generated by $S\dot\cup T'$. 
\begin{dfn}\label{operatorsKR}\normalfont
Let $w\in \hat{F}^+$, such that $w=w_1stw_2$ with $w_1,w_2\in \hat{F}^+$, $s,t\in S$. 
Define the operator $K$ on this expression by
\[
\arraycolsep=2pt\def\arraystretch{1}
K(w;w_1,w_2)=\left\{
\begin{array}{ll}
w_1tsw_2, & \text{if }st=ts \ \text{\normalfont nontrivial in } W,\\
w, & \text{otherwise.}
\end{array}
\right.
\]

Let $w\in \hat{F}^+$, such that $w=w_1sts w_2$ with $w_1, w_2\in \hat{F}^+$, $s,t\in S$.
Define the operator $R$ on this expression by
\[
\arraycolsep=2pt\def\arraystretch{1}
R(w;w_1,w_2)=\left\{
\begin{array}{ll}
w_1tstw_2, & \text{if }sts=tst \ \text{\normalfont nontrivial in } W,\\
w, & \text{otherwise.}
\end{array}
\right.
\]
When an expression of $w$ is specified as above and there is no ambiguity, we would use the notations $K(w)$ for $K(w;w_1,w_2)$ and $R(w)$ for $R(w;w_1,w_2)$.
\end{dfn}

\begin{remark}
	We clearly have $K(w), R(w)\in \hat{F}^+$. Also observe that the operators $K$ and $R$ are analogous to two of the relations in the presentation of the symmetric group $S_n$.
\end{remark}

\begin{proposition}\label{grouprel}

(i)
	Let $w=w_1stw_2$ such that $w_1,w_2\in \hat F^+$ and $s,t\in S$ with $st=ts$ in $W$.	Then for all $i\geq 0$, we have isomorphisms of $\Fpbar$-vector spaces:
	\[
	H^i\left(\overline{X}(w), \Osh_{\overline{X}(w)}\right)\overset{\sim}{\longrightarrow} H^i \left(\overline{X}(K(w)), \Osh_{\overline{X}(K(w))}\right).
	\]
	\par\smallskip
	(ii) Let  $w=w_1sts w_2$ such that $w_1, w_2\in \hat F^+$ and $s,t\in S$ with $sts=tst$ in $W$. For all $i\geq 0$, we have isomorphisms of $\Fpbar$-vector spaces:
	\[
		H^i\left(\Xo(w),\Osh_{\Xo(w)}\right)\overset{\sim}{\longrightarrow} H^i\left(\Xo(R(w)),\Osh_{\Xo(R(w))}\right).
	\] 

	\end{proposition}
\begin{proof}
Let $w$ be as in the assumption of (ii). Recall that $\Xo(w)$ is smooth and projective over $\Fpbar$. Recall from Definition \ref{hatF} that we have $\widehat{sts}\in T'$ because $st\neq ts$. The $\Fpbar$-scheme $ \Xo(w_1\widehat{sts}w_2)$ is projective and smooth by Lemma \ref{XoF+smooth}. In particular, points of $ \Xo(w_1\widehat{sts}w_2)$ are of the form $(B_0',\ldots,B_j', B_{j+1}',\ldots,FB_0')$, where $(B_j',B_{j+1}')\in \overline{O(sts)}$ for some $j$. Now we have a cartesian square:
\begin{equation*}
\begin{tikzcd}
	\Xo(w) \arrow[r]\arrow[d,"f"'] &\overline{O}(s,t,s)\arrow[d]\\
	\Xo(w_1\widehat{sts}w_2)\arrow[r] & \overline{O(sts)},
\end{tikzcd}
\end{equation*}
where the horizontal maps are projections, and the vertical map on the right is the resolution of singularities from \cite[\textsection 9.1]{DL}. Thus the projection map $f$ is proper. 

Observe that the open subscheme $X(w)$ of $\Xo(w)$ is also contained in $\Xo(w_1\widehat{sts}w_2)$ such that the restriction of $f$ to $X(w)$ is the identity. Hence $f$ is a birational morphism. 
By \cite[Theorem 3.2.8]{CR11}, we have for all $i\geq 0$, an isomorphism of $\Fpbar$-vector spaces:
\[
		H^i\left(\Xo(w),\Osh_{\Xo(w)}\right)\overset{\sim}{\longrightarrow} H^i\left(\Xo(w_1\widehat{sts}w_2),\Osh_{\Xo(w_1\widehat{sts}w_2)}\right).
	\] 
	
	For $\widehat{tst}\in T'$, we use the same argument as above to construct a proper birational morphism $f':\Xo(R(w))\to \Xo(w_1\widehat{sts}w_2)$. The key observation is that as $sts=tst$ in $W$, $st\neq ts$, the scheme $\overline{O}(t,s,t)$ gives a smooth compactification of $\overline{O(sts)}$. 
By \cite[Theorem 3.2.8]{CR11}, we have for all $i\geq 0$, an isomorphism of $\Fpbar$-vector spaces:
\[
		H^i\left(\Xo(w_1\widehat{sts}w_2),\Osh_{\Xo(w_1\widehat{sts}w_2)} \right)\overset{\sim}{\longrightarrow} H^i\left(\Xo(R(w)),\Osh_{\Xo(R(w))}\right).
	\] 
This concludes (ii). 
 
 Let $w$ be as in the assumption of (i), and set $\ell(w_1)=r_1$ and $\ell(w_2)=r_2$. For this proof, we define the following projective $\Fpbar$-scheme:
 \begin{multline*}
Y:=\left\{(B_0',\ldots,B_{r_1}', B_{r_1+1}',\ldots,B_{r_1+r_2}', FB_0')\in X^{r_1+r_2+2}\Big\vert (B_{r_1}',B_{r_1+1}')\in \overline{O(st)},\right.\\
 \left. (B_0',\ldots,B_{r_1}')\in \overline{O}(w_1), (B_{r_1+1}',\ldots,B_{r_1+r_2}', FB_0)\in \overline{O}(w_2), \right\}.
 \end{multline*}
  If $s=t$, then the statement is clear. Assume $ s\neq t$ in the following. Since $st= ts$, we see that they are associated to non-adjacent simple reflections. By \cite[Proposition 2.2.16 (iii), (iv)]{DMR}, there exists an isomorphism $\overline{O}(s,t)\to \overline{O(st)}$ such that the restriction to the open subschemes $O(s,t)\to O(st)$ remains an isomorphism. Then $\overline{O(st)}$ is smooth and thus $Y$ is smooth. 
  
  We have an cartesian square:
    \begin{equation*}
  \begin{tikzcd}
  	\Xo(w)\arrow[r]\arrow[d, "f"'] & \overline{O}(s,t)\arrow[d]\\
  	Y \arrow[r] &\overline{O(st)}.
  \end{tikzcd}
  \end{equation*}
  Thus the projection map $f$ on the left is proper. As the restriction of $f$ to $X(w)$ is the identity morphism, we see that $f$ is birational. By \cite[Theorem 3.2.8]{CR11}, for all $i\geq 0$, there is an isomorphism of $\Fpbar$-vector spaces:
  
  	\[
	H^i\left(\overline{X}(w), \Osh_{\overline{X}(w)}\right)\overset{\sim}{\longrightarrow} H^i \left(	Y , \Osh_{	Y}\right).
	\]
  On the other hand, since $st=ts$, we have $O(st)=O(ts)$ and thus $\overline{O(st)}=\overline{O(ts)}$. As in \cite[\textsection 9.1]{DL}, there is a resolution of singularity $\overline{O}(t,s)\to \overline{O(ts)}$. Hence we have a proper birational morphism $\overline{O}(t,s)\to \overline{O(st)}$. Thus we have a cartesian square with horizontal maps being projections:
     \begin{equation*}
  \begin{tikzcd}
  	\Xo(K(w))\arrow[r]\arrow[d, "f'"'] & \overline{O}(t,s)\arrow[d]\\
  	Y \arrow[r] &\overline{O(st)}.
  \end{tikzcd}
  \end{equation*}
Thus $f'$ is proper. Note that the restriction of $f'$ to $X(w_1tsw_2)$ gives an isomorphism between the respective open subschemes $X(w_1tsw_2)$ and $X(w_1stw_2)$ of $\Xo(K(w))$ and $Y$. Thus $f'$ is birational. 
By \cite[Theorem 3.2.8]{CR11}, for all $i\geq 0$, there is an isomorphism of $\Fpbar$-vector spaces:
		\[
	H^i \left(	Y, \Osh_{	Y}\right) \overset{\sim}{\longrightarrow} H^i \left(\overline{X}(K(w)), \Osh_{\overline{X}(K(w))}\right).
	\]
This concludes the proof of (i). 

  \end{proof}

\begin{remark}
	In order to apply \cite[Theorem 3.2.8]{CR11}, we require both schemes in the birational morphism to be smooth over  a perfect field. In fact, this theorem has been generalized to the case where the schemes are allowed to have pseudo-rational singularities \cite[Theorem 8.13]{Kovacs}. In Appendix \ref{AppendixA} we shall see how to use this generalization to show that in the case of arbitrary $G$ as in Section \ref{Gdefn}, the cohomology groups of $\Xo(w)$ and $\overline{X(w)}$  for the structure sheaves (resp. canonical bundles) are isomorphic.  The proposition above then follows from Corollary \ref{cohequiv}.

	Since we will only encounter representations induced from the trivial representation for the proofs of our main results (see Corollary \ref{nonstandardGK} and Proposition \ref{ZpnZ}), we do not need the statement that the isomorphisms in Proposition \ref{grouprel} are $G^F$-equivariant. 
	\end{remark}

\section{The base case}\label{section5}
Let $G=\GL_{n}$. Sections \ref{sws} and \ref{steps} have reduced the study of $H^k\big( \Xo(w),\mathcal O_{\Xo(w)} \big)$ to the case in which $w$ is a Coxeter element or a Coxeter element corresponding to a standard parabolic subgroup $P\subseteq \GL_{n}$. We shall now treat these cases. 

\subsection{Cohomology of $\Xo(w)$ for $w$  a Coxeter element}
We use the notations of Section \ref{GLdfn} and \ref{dhsandmore}. Recall that $\mathbf w$ denotes the standard Coxeter element $s_1\cdots s_{n-1}$.

\begin{proposition}\label{basecase}
 For $k>0$, we have $H^k\big(\Xo(\mathbf w), \Osh_{\Xo(\mathbf w)}\big)=0$. Then the space of global sections
 \[
 H^0\left(\Xo(\mathbf w), \Osh_{\Xo(\mathbf w)}\right)=\Fpbar
 \]
 is the trivial $\GL_n(\Fq)$-representation. 
 \end{proposition}
 
\begin{proof}
Recall that $\Xo(\mathbf w)$ is isomorphic to the successive blow up $\tilde Y$ of $\P^{n-1}_{\Fpbar}$ along all $\Fq$-rational linear subschemes \cite[\textsection 4.1]{Ito} \cite[\textsection 4.1.2]{HWang}. Thus there exists a birational morphism  of $\Fpbar$-schemes $\Xo(\mathbf w)\to \P^{n-1}_{\Fpbar}$. By
 \cite[Theorem 3.2.8]{CR11},
 for all $k\geq 0$, there is an isomorphism of $\Fpbar$-vector spaces:
 \[
 H^k\left(\Xo(\mathbf w), \Osh_{\Xo(\mathbf w)}\right)\overset{\sim}{\longrightarrow}H^k\left(\P^{n-1}_{\Fpbar}, \Osh_{\Xo(\mathbf w)}\right).
 \]
Hence we have
 \[\arraycolsep=2pt\def\arraystretch{1.2}
H^k\left(\Xo(\mathbf w), \Osh_{\Xo(\mathbf w)}\right)=\left\{
\begin{array}{lc}
\Fpbar, & k=0\\
0, & k>0.
\end{array}
\right. 
\]
Therefore it follows that $\GL_n(\Fq)$ acts on  $H^0\big(\Xo(\mathbf w), \Osh_{\Xo(\mathbf w)}\big)$ trivially cf. \textsection \ref{Gequiv}.
 \end{proof}

\begin{remark}
 The above proposition also follows from \cite[Theorem 2.3]{GK} after base change from $\Fq$ to $\Fpbar$.	
\end{remark}

\begin{corollary}\label{nonstandardGK}
Let $w\in F^+$ be an arbitrary Coxeter element. Then the cohomology of $\Xo(w)$ is as follows:
\[\arraycolsep=2pt\def\arraystretch{1.2}
H^k\left(\Xo(w), \Osh_{\Xo(w)}\right)=\left\{
\begin{array}{lc}
\Fpbar, & k=0\\
0, & k>0.
\end{array}
\right.
\]
In particular, $H^0\big(\Xo(w), \Osh_{\Xo( w)}\big)$ is the trivial $\GL_n(\Fq)$-representation. 
\end{corollary}
\begin{proof}

	By \cite[Theorem 3.2.9]{GP}, for any two Coxeter elements $w_1, w_2\in W $, we have $w_1\to w_2$ and $w_2\to w_1$ cf. Definition \ref{cyclicshiftdfn2}. Thus we have $C^k(w')=\mathbf w$ for some integer $k\geq 0$. We may apply Proposition \ref{cyclicshift} to reduce to the case when $w$ is a standard Coxeter element. Use Proposition \ref{basecase} to get the result on the cohomology groups, and it follows that $H^0\big(\Xo(w), \Osh_{\Xo( w)}\big)$ is a trivial $\GL_n(\Fq)$-representation.
\end{proof}

\subsection{Cohomology of $\Xo_{L_I}(w)$ for $w\leq \mathbf w$ and $L_I\subseteq \GL_n$ a standard Levi subgroup}\label{levi}

\begin{lemma}\label{levidecomp}

	Let $w=s_{i_1}\cdots s_{i_{m}}\in F^+$ and $I=\supp (w)$.
	 Then there is an isomorphism of $\Fpbar$-schemes compatible with $L_I(\Fq)$-action:
	\[
	X_{L_I}(w)\overset{\sim}{\longrightarrow} X_{\GL_{n_1}}(w_1)\times \cdots \times  X_{\GL_{n_r}}(w_r), 
	\]
	where $w_a$ is an element in the free monoid associated to the Weyl group of $\GL_{n_a}$, $L_I\cong \prod_{a=1}^r\GL_{n_a}$ and $n=\sum_{a=1}^r n_a$.
		
	When the $s_{i_j}$'s do not repeat, $w_a$ is a Coxeter element in the Weyl group of $\GL_{n_a}$ for all $a$.  In particular, when $w\leq \mathbf{w}$, $w_1\cdots w_r=w$.
\end{lemma}

\begin{proof}
For $a=1,\ldots,r$, denote the Weyl group of $\GL_{n_a}$ by $W_a$. The intersection $L_I\cap B^*$ is a Borel subgroup of $L_I$, and it is a product $B_1\times\cdots \times B_r$, where $B_a$ is a Borel subgroup of $\GL_{n_a}$. Then we have the  homogeneous spaces $X_a:=\GL_{n_a}/B_a$, and the orbit of $v\in W_a$ in $X_a\times X_a$ is $O_a(v)$. 

Note that the Weyl group $W_I$ of $L_I$ is isomorphic to the product of symmetric groups $S_{n_1}\times \cdots \times S_{n_r}$. Hence the Bruhat decomposition of $L_I$ is compatible with the Bruhat decomposition of each  $\GL_{n_a}$, and thus the orbit $O_{L_I}(w)$ is the product $O_1(w_1)\times \cdots \times O_r(w_r)$ over $\normalfont\text{Spec } \Fpbar$. Thus by construction, when $s_{i_j}$ are all distinct, they will each show up exactly once in $w_a$ for exactly one $a$. When $w\leq \mathbf{w}$, we have $w_1\cdots w_r=w$. 

The restriction of the Frobenius endomorphism from $\GL_n$ to $L_I$ respects the product as well. To finish the proof, it suffices to go through the definitions of Deligne--Lusztig varieties with respect to products over  $\normalfont\text{Spec } \Fpbar$. 
\end{proof}
The same applies to the corresponding smooth compactifications with respect to the expression $w=s_{i_1}\cdots s_{i_{m}}$.

\begin{lemma} \label{levidecomp2}
	Using notations as in Lemma \ref{levidecomp}, we have an isomorphism of $\Fpbar$-schemes equivariant under $L_I(\Fq)$-action. 
	
	\[
	\Xo_{L_I}(w)\overset{\sim}{\longrightarrow} \Xo_{\GL_{n_1}}(w_1)\times \cdots \times  \Xo_{\GL_{n_r}}(w_r)
	\]
\end{lemma}
\begin{proof}
	It suffices check the definitions of smooth compactifications for Deligne--Lusztig varieties with respect to the product of reductive groups. 
\end{proof}

\begin{remark}
	 If $\supp(w)$ is a proper subset of $S$, then sometimes we could have $\GL_{n_a}=\GL_1$ for some $a$, thus $\Xo_{\GL_{n_a}}(e)$ is the point corresponding to the only Borel subgroup.
\end{remark}

\begin{proposition} \label{levicoh}
 Let $w\leq \mathbf{w}$ and $I=\supp(w)$. Then 
\[
H^k\left(\Xo_{L_I}(w),\mathcal O_{\Xo_{L_I}(w)}\right)=0,
\] for $k>0$,  and 
\[
H^0\left(\Xo_{L_I}(w),\mathcal O_{\Xo_{L_I}(w)}\right)=\Fpbar
\]
is the trivial $L_I(\Fq)$-representation.
\end{proposition}
\begin{proof}

	By Lemma \ref{levidecomp2}, we may compute the cohomology for $\Xo_{\GL_{n_1}}(w_1)\times \cdots \times  \Xo_{\GL_{n_r}}(w_r)$, with notations as before. To simplify the notation, set 
				\[
	V_j:= \Xo_{\GL_{n_1}}(w_1)\times \cdots \times  \Xo_{\GL_{n_j}}(w_j).
	\]
	By applying induction on the K\"{u}nneth formula for coherent sheaves \cite[Theorem  6.7.8]{EGA32}, we have
	\[
	H^k\left( V_j, \mathcal O_{V_j}\right)=\bigoplus_{p+q=k} H^p\left( V_{j-1},\mathcal O_{V_{j-1}}\right)\otimes_{\Fpbar} H^q\left(\Xo_{\GL_{n_j}}(w_j), \mathcal O_{\Xo_{\GL_{n_j}}(w_j)}\right).
	\]
	By Proposition \ref{basecase}, we know that for all $j=1,\ldots,r$, 
	\[\arraycolsep=2pt\def\arraystretch{1.2}
	H^k\left(\Xo_{\GL_{n_j}}(w_j), \mathcal O_{\Xo_{\GL_{n_j}}(w_j)}\right)=\left\{
	\begin{array}{lc}
\Fpbar, & k=0\\
0, & k>0.
\end{array}
\right.
	\]
Thus, 
\[
H^k\left( V_j, \mathcal O_{V_j}\right)= H^k\left( V_{j-1}, \mathcal O_{V_{j-1}}\right)\otimes_{\Fpbar} \Fpbar.
\]
Hence by induction on $j$ we know that for any $j$, 
\[\arraycolsep=2pt\def\arraystretch{1.2}
H^k\left( V_j, \mathcal O_{V_j}\right)=\left\{
\begin{array}{lc}
\Fpbar, & k=0\\
0, & k>0.
\end{array}
\right.
	\]
	The case $j=r$ yields the desired result. It follows from 
		\[
H^0\left(\Xo_{L_I}(w),\mathcal O_{\Xo_{L_I}(w)}\right)=\Fpbar, 
\]
that $H^0\big(\Xo_{L_I}(w),\mathcal O_{\Xo_{L_I}(w)}\big)$ is the trivial $L_I(\Fq)$-representation cf. \textsection \ref{Gequiv}.
\end{proof}
\begin{corollary}\label{levicoh2}
	Let $w=s_{i_1}\cdots s_{i_{m}}\in F^+$ and $I=\normalfont\text{supp}(w)$, such that the $s_{i_j}$'s are all distinct.  Then  
		\[\arraycolsep=2pt\def\arraystretch{1.2}
	H^k\left(\Xo_{L_I}(w),\mathcal O_{\Xo_{L_I}(w)}\right)=\left\{
	\begin{array}{lc}
\Fpbar, & k=0\\
0, & k>0.
\end{array}
\right.
	\]
Furthermore, $H^0\left(\Xo_{L_I}(w),\mathcal O_{\Xo_{L_I}(w)}\right)$	is the trivial $L_I(\Fq)$-representation.  \end{corollary}
\begin{proof}
	This is analogous to the proof of Proposition \ref{levicoh} and uses Corollary \ref{nonstandardGK}. 
\end{proof}

We conclude this section by recalling the following lemma.

\begin{lemma}
Let $w=s_{i_1}\cdots s_{i_{m}}\in F^+$ and $I=\supp (w)$.  Then the irreducible components of $X(w)$ (resp. $\Xo(w)$) are $\vert G(\mathbb F_q)/P_I(\mathbb F_q)\vert $ isomorphic copies of $X_{L_I}(s_{i_1},\ldots,s_{i_{m}})$ (resp. $\Xo_{L_I}(s_{i_1},\ldots,s_{i_{m}})$).
\end{lemma}

\begin{remark}
We see from this lemma that for any $w\in F^+$, after fixing a reduced expression, the number of irreducible components of $X(w)$ (resp. $\Xo(w)$) depends only on $\supp(w)$. 
\end{remark}

\section{The main Theorem}\label{theorem1}
\subsection{Cohomology of the structure sheaf on $\Xo(w)$}

We now prove our main theorem.

\begin{proof}[Proof of Theorem \ref{main}]
After fixing a support of $w$, we may apply the induction steps. First, suppose that $\supp (w)=S$. 
If $w=s_{i_1}\cdots s_{i_m}\in F^+$ and   the $s_{i_j}$'s are not all distinct, then we apply Proposition \ref{cyclicshift} and Proposition \ref{grouprel} to transform $w$ into the shape $sw's$ with $s\in S$. We use Proposition \ref{p1coh} to reduce its length. After repeating this procedure of changing the presentation of $w$ and reducing the length finitely many times, we have for all $k\geq 0$, an isomorphism of $\Fpbar$-vector spaces:
	\[
	H^k\left(\Xo(w), \Osh_{\Xo(w)}\right)\overset{\sim}{\longrightarrow}H^k\left(\Xo(v), \Osh_{\Xo(v)}\right),
	\]
	where $v\in F^+$ corresponds to a Coxeter element and $\normalfont\text{supp}(v)=\text{supp}(w)$. By Corollary \ref{nonstandardGK}, we know that for $k>0$, $H^k\big(\Xo(w), \Osh_{\Xo(w)}\big)$ vanish, and
	\[
	H^0\left(\Xo(w), \Osh_{\Xo(w)}\right)=\Fpbar.
	\]
	Thus we know that $H^0\big(\Xo(w), \Osh_{\Xo(w)}\big)$ is the trivial  $\GL_n(\Fq)$-representation cf. \textsection \ref{Gequiv}.

Next, let $I:=\supp(w)$ such that  $I\subsetneq S$.
If $w=s_{i_1}\cdots s_{i_m}\in F^+$ and the $s_{i_j}$'s are not all distinct, then we apply
Proposition \ref{cyclicshift} and Proposition \ref{grouprel} to transform $w$ into the shape $sw's$ with $s\in S$. We use Proposition \ref{p1coh} to reduce its length. After repeating this procedure finitely many times, we have for all $k\geq 0$, an isomorphism of $\Fpbar$-vector spaces:
	\[
	H^k\left(\Xo(w), \Osh_{\Xo(w)}\right)\overset{\sim}{\longrightarrow}H^k\left(\Xo(v), \Osh_{\Xo(v)}\right),
	\]
	where $v\in F^+$ corresponds to a Coxeter element in $W_I$ and $\normalfont\text{supp}(v)=\text{supp}(w)$. 
	By \cite[Proposition 2.3.8]{DMR}, we have an $\GL_n(\Fq)$-equivariant isomorphism of $\Fpbar$-schemes: 
	\[
	\Xo(v)\overset{\sim}{\longrightarrow}  \GL_n^F/U_I^F\times^{L_I^F}\Xo_{L_I}(v).
	\]
	Precomposing this with the morphism 
	\begin{equation*}
\arraycolsep=1pt\def\arraystretch{1.5}
		\begin{array}{rcl}
		\pi_I: \GL_n^F/U_I^F\times^{L_I^F}\Xo_{L_I}(v) & \longrightarrow & \GL_n^F/P_I^F\\
		\left(xU_I^F, \left(B_0,\ldots,B_{m'}\right)\right) &\longmapsto &  \pi(xU_I^F),
\end{array}
\end{equation*}
where $\pi: \GL_n^F/U_I^F\to \GL_n^F/P_I^F$ is the natural projection map, 
	gives a $\GL_n(\Fq)$-equivariant morphism $\Xo(v)\to \GL_n^F/P_I^F$ 
such that $\GL_n(\Fq)$ acts on the set of fibres transitively and the stablizer of each fibre corresponds to a conjugate of $P_I(\Fq)$ in $\GL_n(\Fq)$. 	It follows that 
	\[
	H^k\left(\Xo(v), \Osh_{\Xo(v)}\right)= \Ind_{P_I(\Fq)}^{\GL_n(\Fq)} H^k\left(\Xo_{L_I}(v), \Osh_{\Xo_{L_I}(v)}\right),
	\]
	for all $k\geq 0$. 
	Thus Corollary \ref{levicoh2} implies the vanishing of $H^k\left(\Xo(w), \Osh_{\Xo(w)}\right) $ for $k>0$.
	
	Finally, we analyze the global sections of $\Osh_{\Xo(w)}$ as a $\GL_n(\Fq)$-representation.  By Lemma \ref{levidecomp2}, we know that 
	\[\Xo_{L_I}(w)\overset{\sim}{\longrightarrow}\Xo_{\GL_{n_1}}(w_1)\times \cdots \times  \Xo_{\GL_{n_r}}(w_r),
	\]
	is an isomorphism of $\Fpbar$-schemes equivariant under $L_I(\Fq)$-action, where $n_1+\cdots +n_r=n$ and $w_a$ is an element in the Weyl group of $\GL_{n_a}$, $a=1,\ldots,r$. This gives an isomorphism of $\GL_n(\Fq)$-modules:	
	\begin{multline*}
		H^0\left(\Xo_{L_I}(w), \Osh_{\Xo_{L_I}(w)}\right)\overset{\sim}{\longrightarrow} \\
		H^0\left( \Xo_{\GL_{n_1}(w_1)}, \mathcal O_{\Xo_{\GL_{n_1}(w_1)}}\right)\otimes\cdots\otimes H^0\left( \Xo_{\GL_{n_r}}(w_r), \mathcal O_{\Xo_{\GL_{n_r}}(w_r)}\right).
	\end{multline*}
		Since $w$ has full support in $W_I$, we know that $w_a$ has full support in $W_a$ for all $a=1,\ldots,r$, where $W_a$ is the Weyl group of $\GL_{n_a}$. It follows from Corollary \ref{nonstandardGK} that
		\[
		H^0\left( \Xo_{\GL_{n_a}}(w_a), \mathcal O_{\Xo_{\GL_{n_a}}(w_a)}\right)=\Fpbar
		\]
	is the trivial representation for $\GL_{n_a}(\Fq)$, so we know that  
	\[
	H^0\left(\Xo_{L_I}(w), \Osh_{\Xo_{L_I}(w)}\right)=\Fpbar
	\]
	gives the trivial $L_I(\Fq)$-representation. Therefore
	\[
	H^0\left(\Xo(w), \Osh_{\Xo(w)}\right)=\Ind_{P_I(\Fq)}^{\GL_n(\Fq)} \mathds{1}_{\Fpbar},
	\]
	where $\mathds{1}_{\Fpbar}$ is the trivial $P_I(\Fq)$-representation with coefficients in $\Fpbar$.
	\end{proof}

	\begin{remark}
			For $w\in W$, if we fix a reduced expression $w=s_{i_1}\cdots s_{i_m}$ with $s_{i_j}\in S$, then $s_{i_1}\cdots s_{i_m}$ can be considered as an element of $F^+$. Theorem \ref{main} thus applies to $\Xo(w)$ via the $\GL_n(\Fq)$-equivariant isomorphism $\Xo(s_{i_1},\ldots, s_{i_m})\cong \Xo(w)$. 
	\end{remark}

\subsection{The ${\rm mod}\ p^m$ and $\mathbb Z_p$ \'{e}tale cohomology of $\Xo(w)$}
For the rest of this paper, if the finite group $H$ is clear from the context and $R$ is a ring, we fix the notation $\mathds{1}_R$ for the free $1$-dimensional trivial $H$-module with coefficients in $R$.

Let $X$ be a $k$-scheme with $k$ being a field of characteristic $p>0$. We have the constant sheaf $\ZZ/p\ZZ$ on  $X_\et$. Note that the associated presheaf of the group scheme $\mathbb G_a$ is a sheaf on both $X_\et$ and $X_{\Zar}$ \cite[p. 52]{MilneEC}. In particular, it gives the structure sheaf on $X_{\Zar}$.  Recall the Artin--Schreier sequence \cite[p. 67]{MilneEC}.

\begin{lemdfn}
\normalfont
Let $X$ be a $k$-scheme with $k$ being a field of characteristic $p>0$. Let $F_p$ be the $p$-Frobenius on $\Osh_X$ sending $x\mapsto  x^p$. There exists a short exact sequence of sheaves on $X_\et$:
\[
0 \rightarrow \ZZ/p\ZZ \longrightarrow \mathbb{G}_{a}\overset{F_p-1}{\longrightarrow} \mathbb{G}_{a}\rightarrow 0 .
	\]
We call this the \emph{Artin--Schreier sequence}.
\end{lemdfn}
 
 Note that by \cite[p. 114]{MilneEC}, the cohomology of $\mathbb G_a$ on $X_\et$ and $X_{\Zar}$ are isomorphic.

\begin{remark}
Recall that for $G=\GL_n$, we have
	\[\arraycolsep=2pt\def\arraystretch{1.2}
H^0\left(\Xo(w), \Osh_{\Xo(w)}\right)=\Ind_{P_I(\Fq)}^{\GL_n (\Fq)}\Fpbar,
\]
where $I=\normalfont\text{supp}(w)$ and $P_I=B^*W_IB^*$. 
Note that $p$-Frobenius $F_p$ on $\Xo(w)$ is given by the identity on the topological space and $p$-power map on sections of $\Osh_{\Xo(w)}$. Thus the $p$-power map given by $F_p$ on the global sections $H^0\big(\Xo(w), \Osh_{\Xo(w)}\big)$ on the left hand side is induced by the $p$-power map on $\Fpbar$.
\end{remark}

\begin{proposition}\label{ZpnZ}
	Let $G=\GL_n$, and $w\in F^+$ with $I=\normalfont\text{supp}(w)$ and $P_I=B^*W_IB^*$. For any integer $m>0$, we have 
	\[\arraycolsep=2pt\def\arraystretch{1.2}
\Het^k\left(\Xo(w), \mathbb Z/p^m\mathbb Z\right)=\left\{
\begin{array}{lc}
\Ind_{P_I(\Fq)}^{\GL_n(\Fq)} \mathds{1}_{\mathbb Z/p^m\mathbb Z}, & k=0,\\
0, & k>0.
\end{array}
\right.
\]
\end{proposition}

\begin{proof}
Let $m=1$, and consider
 the long exact sequence associated to the Artin--Schreier sequence on $\Xo(w)$.
\[\arraycolsep=1pt\def\arraystretch{2}
\begin{array}{c}
0\to H^0_{\text{\'{e}t}}\big(\Xo(w), \ZZ/p\ZZ \big) \to H^0\big(\Xo(w), \Osh_{\Xo(w)}\big) \overset{F_p-1}{\longrightarrow} H^0\big(\Xo(w), \Osh_{\Xo(w)}\big)\to
\\  \to H^1_{\text{\'{e}t}}\big(\Xo(w), \ZZ/p\ZZ \big) \to H^1\big(\Xo(w), \Osh_{\Xo(w)}\big)  \overset{F_p-1}{\longrightarrow} H^1\big(\Xo(w), \Osh_{\Xo(w)}\big)  \to \cdots.
\end{array}
\]
By Theorem \ref{main}, we know that 
\[H^k\big(\Xo(w), \Osh_{\Xo(w)}\big) =0,\]
 for all $k>0$,  and 
 \[H^0\big(\Xo(w), \Osh_{\Xo(w)}\big)=\Ind_{P_I(\Fq)}^{\GL_n (\Fq)}\Fpbar.\]
 As a consequence,  the long exact sequence above becomes
\begin{multline*}
	0\to H^0_{\text{\'{e}t}}\big(\Xo(w), \ZZ/p\ZZ \big)  \to  H^0\big(\Xo(w), \Osh_{\Xo(w)}\big)  \overset{F_p-1}{\longrightarrow} 
	\\
	H^0\big(\Xo(w), \Osh_{\Xo(w)}\big)  
	\to  H^1_{\text{\'{e}t}}\big(\Xo(w), \ZZ/p\ZZ \big)\to 0.
\end{multline*}

Since $\Fpbar$ is algebraically closed, the polynomial $x^p-x\in \Fpbar[x]$ always splits. Note that $\Ind_{P_I(\Fq)}^{\GL_n (\Fq)}\Fpbar$ is a finite dimensional $\Fpbar$-vector space, so $F_p-1$ is the map $x\mapsto x^p-x$ on each coordinate. Hence $F_p-1$ is surjective. Therefore $H^1_{\text{\'{e}t}}\big(\Xo(w), \ZZ/p\ZZ \big)=0$.

For every integer $m\geq 2$, we have the short exact sequence
\[
0\to \mathbb Z/p^{m-1}\mathbb Z\overset{\alpha}{\longrightarrow}\mathbb Z/p^m\mathbb Z\overset{\beta}{\longrightarrow}\ZZ/p\ZZ\to 0,
\]
where 
\[\arraycolsep=1pt\def\arraystretch{1.5}
\begin{array}{rcl}
	\alpha: a \  {\rm mod} \ p^{m-1} &\longmapsto & pa \ {\rm mod} \ p^m\\
	\beta: b \ {\rm mod} \ p^m &\longmapsto & b \ {\rm mod} \ p. 
\end{array}
\]
We get an induced long exact sequence for every $m\geq 2$, 
	 \begin{equation*}
	 	 \arraycolsep=1pt\def\arraystretch{2}
	 \begin{array}{c}
	 	0\to H^0_{\text{\'{e}t}}\big(\Xo(w), \ZZ/p^{m-1}\ZZ \big) \to H^0_{\text{\'{e}t}}\big(\Xo(w), \ZZ/p^{m}\ZZ \big) \to H^0_{\text{\'{e}t}}\big(\Xo(w), \ZZ/p\ZZ \big) \to \cdots \\
	 	\cdots \to H^k_{\text{\'{e}t}}\big(\Xo(w), \ZZ/p^{m-1}\ZZ \big) \to  H^k_{\text{\'{e}t}}\big(\Xo(w), \ZZ/p^{m}\ZZ \big) \to H^k_{\text{\'{e}t}}\big(\Xo(w), \ZZ/p\ZZ \big)\to \cdots .
	 		 \end{array}
	 \end{equation*}
	 We know that $H^k_{\text{\'{e}t}}\big(\Xo(w), \ZZ/p\ZZ \big)=0$ for all $k>0$. By induction on $m$, assume that $H^k_{\text{\'{e}t}}\big(\Xo(w), \ZZ/p^{m-1}\ZZ \big)=0$ for all $k>0$, so the long exact sequence yields that $H^k_{\text{\'{e}t}}\big(\Xo(w), \ZZ/p^m\ZZ \big)=0$ for all $k>0$. Therefore for any integer $m>0$, we have
	 \[H^k_{\text{\'{e}t}}\big(\Xo(w), \ZZ/p^m\ZZ \big)=0
	 \]
	  for all $k>0$.

	 For any commutative ring $A$ of characteristic $p$, denote $W_j(A)$ to be the ring of Witt vectors of length $j$ with coefficients in $A$.  Recall that $W_j(A)$ is set-theoretically in bijection with the product $A^j$, but the bijection is not an isomorphism of rings when $j>1$. However, the $p$-Frobenius $F_p$ on $W_j(A)$ is compatible with the $p$-Frobenius on $A$, in the sense that $F_p$ on $W_j(A)$ is the map $x\mapsto x^p$ on each coordinate. See \cite[\textsection 0.1]{Illusie} for an introduction on the ring of Witt vectors. 
	  
	 For any integer $m>0$, let $W_m\left(\mathcal O_{\Xo(w)}\right)$ be the sheaf of Witt vectors of length $m$ on $\Xo(w)$ \cite[\textsection 0.1.5]{Illusie}. The stalk of $W_m\left(\mathcal O_{\Xo(w)}\right)$ at a point $x\in X$ is $W_m\left(\mathcal O_{\Xo(w), x}\right)$ \cite[(01.5.6)]{Illusie}. Note that $W_m\left(\mathcal O_{\Xo(w)}\right)$ are coherent sheaves on $\Xo(w)$ \cite[\textsection 2]{SerreWitt}. Similar to the ring of Witt vectors, the sections of the coherent sheaf $W_m\left(\mathcal O_{\Xo(w)}\right)$ are set-theoretically in bijection with the corresponding sections of $\mathcal O_{\Xo(w)}^m$, but the bijection is not an isomorphism of rings when $m>1$. The $p$-Frobenius $F_p$ on $W_m\left(\mathcal O_{\Xo(w)}\right)$ is compatible with the $p$-Frobenius on $\mathcal O_{\Xo(w)}$, in the sense that it is $x\mapsto x^p$ on each coordinate.
	  
	 On $\Xo(w)_{\et}$, we have the \emph{Artin--Schreier--Witt exact sequence}, cf. \cite[Proposition 3.28]{Illusie},
	 \[
	 0\longrightarrow \mathbb Z/p^m\mathbb Z\to W_m\left(\Osh_{\Xo(w)}\right)\overset{F_{p}-1}{\longrightarrow} W_m\left(\Osh_{\Xo(w)}\right)\to 0.
	 \]
	 One attains the long exact sequence 
	 \begin{multline*}
	 	0\longrightarrow H^0\left(\Xo(w), \mathbb Z/p^m\mathbb Z\right) \to H^0\left(\Xo(w), W_m\left(\Osh_{\Xo(w)}\right)\right)\overset{F_{p}-1}{\longrightarrow}\\
	 	 H^0\left(\Xo(w), W_m\left(\Osh_{\Xo(w)}\right)\right)\to \cdots .
	 \end{multline*}
	 We see that $H^0\left(\Xo(w), \mathbb Z/p^j\mathbb Z\right)=\ker(F_p-1)$. 
	 
	 First, let $\supp(w)=S$.  Since $\Xo(w)$ is smooth projective over $\Fpbar$ and 
	 \[H^0\left(\Xo(w), \mathcal O_{\Xo(w)}\right)=\Fpbar,\]
	  we know that $H^0\left(\Xo(w), W_m\left(\Osh_{\Xo(w)}\right)\right)= W_m\left(\Fpbar\right)$. Note that the $p$-Frobenius $F_p$ on $W_m\left(\Osh_{\Xo(w)}\right)$ is compatible with taking sections, so it is compatible with the Frobenius on $W_m\left(\Fpbar\right)$. 	 
	 Thus we have 
	 \[\ker(F_p-1)=W_m(\mathbb F_p)\]
	  for all $m>0$. We know that $W_m(\mathbb F_p)=\mathbb Z/p^m\mathbb Z$ for all $m>0$. Since $\GL_n(\Fq)$ acts trivially on $H^0\left(\Xo(w), \mathcal O_{\Xo(w)}\right)=\Fpbar$ and $\GL_n(\Fq)$ acts
	  trivially on each of the coordinate of $W_m\left(\Fpbar\right)$, we see that 
	   $W_m\left(\Fpbar\right)$ inherits a trivial $\GL_n(\Fq)$-action.
	   	 	 
	 Now let $w\in W$ be an arbitrary element and set $I=\supp(w)$.  Via \cite[Proposition 2.3.8]{DMR} (cf. Theorem \ref{main}) we have a $\GL_n(\Fq)$-equivariant surjective morphism
\[
\Xo(w) \longrightarrow \GL_n^F/P_I^F
\]
such that the fibres are $\Xo_{L_I}(w)$. In particular, $\GL_n(\Fq)$ acts on the set of fibres transitively and the stablizer of each fibre corresponds to a conjugate of $P_I(\Fq)$ in $\GL_n(\Fq)$. 
	 
	 	By definition one has
	 
	 \[
	 W_m\left(\Osh_{\Xo(w)}\right)\left(\Xo(w)\right)=W_m\left(\Osh_{\Xo(w)}\left(\Xo(w)\right)\right),
	 \]
	 and since the functor $W_m$ is a finite limit, we have
	 \[
	 W_m\left(\Osh_{\Xo(w)}\left(\Xo(w)\right)\right)=\Ind_{P_I(\Fq)}^{\GL_n (\Fq)} W_m\left(\Osh_{\Xo_{L_I}(w)}\left(\Xo_{L_I}(w)\right)\right).
	 \]
	It follows from the proof of Theorem \ref{main} that 	$\Osh_{\Xo_{L_I}(w)}\left(\Xo_{L_I}(w)\right)=\Fpbar$,     so $\Osh_{\Xo_{L_I}(w)}\left(\Xo_{L_I}(w)\right)$ is the trivial $P_I(\Fq)$-representation. As before, this makes $W_m\left(\Osh_{\Xo_{L_I}(w)}\left(\Xo_{L_I}(w)\right)\right)$ the 1-dimensional trivial $P_I(\Fq)$-module. Thus we have
	
	\[
	\ker(F_p-1)=\Ind_{P_I(\Fq)}^{\GL_n (\Fq)} W_m(\Fp)=\Ind_{P_I(\Fq)}^{\GL_n (\Fq)} \mathds{1}_{\mathbb Z/p^m\mathbb Z}
	\]
	where $\mathds{1}_{\mathbb Z/p^m\mathbb Z}$ is the trivial $P_I(\Fq)$-representation with coefficients in $\mathbb Z/p^m\mathbb Z$. This finishes the proof.
	 		 \end{proof}

	 		 \begin{corollary}\label{Zpxo}
	Let $G=\GL_n$, and $w\in F^+$ with $I=\normalfont\text{supp}(w)$. Let $P_I=B^*W_IB^*$. Then one has
	\[\arraycolsep=2pt\def\arraystretch{1.2}
\Het^k\left(\Xo(w), \mathbb Z_p\right)=\left\{
\begin{array}{lc}
\Ind_{P_I(\Fq)}^{\GL_n(\Fq)} \mathds{1}_{\mathbb Z_p}, & k=0,\\
0, & k>0.
\end{array}
\right.
\]
\end{corollary}
\begin{proof}
By Proposition \ref{ZpnZ}, since for all $k>0$ and $m>0$, we have $H_\et^k\big(\Xo(w), \ZZ/p^m\ZZ \big)=0$, the tower $\left\{H_\et^k\left(\Xo(w),\mathbb Z/p^m\mathbb Z\right)\right\}_m$ of abelian groups satisfy the Mittag--Leffler condition trivially. Thus for all $k>0$, we have 
	 \[H^k_{\text{\'{e}t}}\big(\Xo(w), \ZZ_p \big)=\varprojlim_{m}H^k_{\text{\'{e}t}}\big(\Xo(w), \ZZ/p^m\ZZ \big)=0.
	 \]
On the other hand, the higher vanishing implies that  whenever we have  $m>l$ and a ${\rm mod}\ p^{l}$ map
\[\arraycolsep=1pt\def\arraystretch{1.5}
\begin{array}{rcl}
	\mathbb Z/p^m\mathbb Z &\longrightarrow &\mathbb Z/p^l\mathbb Z \\
	b\ {\rm mod} \ p^m &\longmapsto & b \ {\rm mod} \ p^l,
\end{array}
\]
Take the induced short exact sequence of sheaves on $\Xo(w)$:
\[
0\to \ZZ/p^{m-l}\ZZ\to \ZZ/p^m\ZZ\to \ZZ/p^l\ZZ\to 0.
\]
Take the associated long exact sequence of cohomology groups
\[\cdots\to 
H^0_{\text{\'{e}t}}\big(\Xo(w), \ZZ/p^{m}\ZZ \big)\to H^0_{\text{\'{e}t}}\big(\Xo(w), \ZZ/p^{l}\ZZ \big)\to H^1_{\text{\'{e}t}}\big(\Xo(w), \ZZ/p^{m-l}\ZZ \big)\to\cdots.
\]
By the higher vanishing from Proposition \ref{ZpnZ}, we see that the morphism 
\[
H^0_{\text{\'{e}t}}\big(\Xo(w), \ZZ/p^{m}\ZZ \big)\longrightarrow H^0_{\text{\'{e}t}}\big(\Xo(w), \ZZ/p^{l}\ZZ \big)
\]
is surjective. 
Therefore the tower $\{H_\et^0\left(\Xo(w),\mathbb Z/p^m\mathbb Z\right)\}_m$ of abelian groups satisfies the Mittag--Leffler condition. Thus we have
  \[
	 H_{\et}^0\left(\Xo(w),\mathbb Z_p\right)=\varprojlim_m H_{\et}^0\left(\Xo(w), \mathbb Z/p^m\mathbb Z\right),
	 \]
	and the identification 
	\[
	H^0_{\text{\'{e}t}}\left(\Xo(w), \mathbb Z_p\right)=\Ind_{P_I(\Fq)}^{\GL_n(\Fq)} \mathbb Z_p.
	\]
 \end{proof}

\subsection{Cohomology of $\Omega^{\ell(w)}$ on $\Xo(w)$}
Let $G=\GL_n$ and $w\in W$. Recall that $\Xo(w)$ is smooth projective of dimension $\ell(w)$. Let $\Omega$ be the sheaf of differentials on $\Xo(w)$ and write $\Omega^p=\wedge^p \Omega$ for the sheaf of differential $p$-forms. In particular, $\omega:=\wedge^{\ell(w)}\Omega=\Omega^{\ell(w)}$ is the dualizing sheaf on $\Xo(w)$.

\begin{proposition}
Let $G=\GL_n, w\in F^+$, $I=\supp(w)$ and $P_I=B^*W_IB^*$, then
	\[\arraycolsep=2pt\def\arraystretch{1.2}
H^k\left(\Xo(w), \Omega^{\ell(w)}\right)=\left\{
\begin{array}{lc}
\Ind_{P_I(\Fq)}^{\GL_n(\Fq)}  \mathds{1}_{\Fpbar}, & k=\ell(w),\\
0, & k\neq \ell(w).
\end{array}
\right.
\]
\end{proposition}
 
\begin{proof}
Since $\Xo(w)$ is a smooth projective $\Fpbar$-scheme of dimension $\ell(w)$, Serre duality implies that there is an isomorphism of $\Fp$-schemes 
\[
H^q\big(\Xo(w),\Omega^p\big)\overset{\sim}{\longrightarrow} H^{\ell(w)-q} \big(\Xo(w), \Omega^{\ell(w)-p}\big)^\vee
\]
for all $p, q\geq 0$. Fix $p=0$ and we may apply Theorem \ref{main} to get
\[
H^{k} \big(\Xo(w), \Omega^{\ell(w)}\big)=0,
\]
when $k\neq \ell(w)$. Setting $p=0$ and $q=0$, we get by Theorem \ref{main} that 
\[
H^0\big(\Xo(w),\Osh_{\Xo(w)}\big)\overset{\sim}{\longrightarrow} H^{\ell(w)} \big(\Xo(w), \Omega^{\ell(w)}\big)^\vee.
\]
 This isomorphism is equivariant under $\GL_n(\Fq)$-action by \cite[Theorem 29.5]{Hashimoto}. Finally, recall that since $\GL_n(\Fq)$ is a finite group, the induction functor commutes with taking the dual of a representation. 
\end{proof}

\section{The compactly supported ${\rm mod} \ p^m$ and $\mathbb Z_p$ \'{e}tale cohomology of $X(w)$}\label{corollaryX(w)}

Throughout this section, let $G=\GL_n$ and $w\in W$ as in Section \ref{GLdfn}. Fix a reduced expression $w=t_1\cdots t_r,$ $t_j\in S$. This reduced expression determines a smooth compactification $\Xo(w)$ for $X(w)$. We also have an  isomorphism  $X\left(t_1,\ldots, t_r\right)\overset{\sim}{\rightarrow}X(w)$ cf. Remark \ref{isoofgeneralizeddlvar}. By \cite[Proposition 3.2.2]{DMR}, we have the following decomposition:
\[
\Xo(w):=\Xo\left(t_1,\ldots, t_r\right)=X\left(t_1,\ldots, t_r\right)\bigcup^{\boldsymbol{\cdot} } \left( \bigcup_{\substack{v\prec w\\ \ell(v)=\ell(w)-1}} \Xo(v)\right),
\]
where $\prec$ is the Bruhat order on $ F^+$.
Let us denote $Y:= \Xo(w)\Big\backslash X\left(t_1,\ldots, t_r\right)$ in this section. We use this stratification and Proposition \ref{ZpnZ} to obtain the compactly supported $\mathbb \ZZ/p^m\mathbb Z$-cohomology of $X(w)$. 
In order to do this, we construct an exact sequence similar to Mayer--Vietoris spectral sequence with respect to the stratification of $Y$. 
 The method has been used to compute the compactly supported $\ell$-adic cohomology groups of $X(w)$ in
   \cite[\textsection 5, \textsection 7]{Orlik18}.

\subsection{An acyclic resolution for the Steinberg representation for a Levi subgroup of $\GL_n$}\label{acyclicity}	 

Let $w\in W$ such that  $w=t_1\cdots t_r$ with $t_{j}\in S$ are all distinct from one another. We have the associated parabolic subgroup $P_I=B^*W_IB^*$, where $I=\supp(w)$. Set $I_u=\supp(u)$ for  $u\preceq w$. Consider the following sequence:
	\begin{multline}\label{stseq}
	\Ind_{P_I(\Fq)}^{\GL_n(\Fq)}\mathds{1}_{\mathbb Z/p^m\mathbb Z}\overset{d_0}{\longrightarrow} \bigoplus_{\substack{u\prec w\\ \ell(u)=\ell(w)-1}}\Ind_{P_{I_u}(\Fq)}^{\GL_n(\Fq)}\mathds{1}_{\mathbb Z/p^m\mathbb Z}\to\cdots \\
		\cdots \to \bigoplus_{\substack{u\prec w \\\ell(u)=\ell(w)-i+1}}\Ind_{P_{I_{u}}(\Fq)}^{\GL_n(\Fq)}\mathds{1}_{\mathbb Z/p^m\mathbb Z}\overset{d_{i-1}}{\longrightarrow} 
		\bigoplus_{\substack{u\prec w\\ \ell(u)=\ell(w)-i}}\Ind_{P_{I_u}(\Fq)}^{\GL_n(\Fq)}\mathds{1}_{\mathbb Z/p^m\mathbb Z}\overset{d_i}{\longrightarrow}\\
		\overset{d_i}{\longrightarrow} \bigoplus_{\substack{u\prec w\\ \ell(u)=\ell(w)-i-1}}\Ind_{P_{I_u}(\Fq)}^{\GL_n(\Fq)}\mathds{1}_{\mathbb Z/p^m\mathbb Z}\to\cdots \\
		\cdots \to \bigoplus_{\substack{u\prec w\\ \ell(u)=1}}\Ind_{P_{I_u}(\Fq)}^{\GL_n(\Fq)}\mathds{1}_{\mathbb Z/p^m\mathbb Z}\overset{d_{\ell(w)-1}}{\longrightarrow} 
	 	\Ind_{B^*(\Fq)}^{\GL_n(\Fq)}\mathds{1}_{\mathbb Z/p^m\mathbb Z}.
	\end{multline}
	For all $u_{i+1}\preceq u_{i}\preceq w$ with $\ell(u_{i+1})=\ell(u_i)-1$, let \[\iota_{u_i}^{u_{i+1}}:\Ind_{P_{u_i}(\Fq)}^{\GL_n(\Fq)}\mathds{1}_{\mathbb Z/p^m\mathbb Z}\to \Ind_{P_{u_{i+1}}(\Fq)}^{\GL_n(\Fq)}\mathds{1}_{\mathbb Z/p^m\mathbb Z}\]  be the inclusion map, where $P_{u_i}:=P_{\supp(u_i)}$. 	
	Then the map $d_i$ is defined by 
	\[
	(f_{u_i})_{u_i}\mapsto \left(\sum_{u_i}(-1)^{\alpha(u_i\to u_{i+1})} \iota_{u_i}^{u_{i+1}}(f_{u_i})\right)_{u_{i+1}},
	\]
	where $\alpha(u_i\to u_{i+1})$ is an integer determined as follows: if $u_{i+1}$ is obtained from $u_i$ by deleting the $r$-th $s\in S$ in its product expression, then $\alpha(u_i\to u_{i+1})=r$. 

	We will see in the following proposition that this sequence is an acyclic complex. In particular, if $w\in W$ is a Coxeter element, then (up to augmentation) the complex (\ref{stseq}) gives a resolution for the Steinberg representation $\St_{\GL_n}$.

\begin{proposition}\label{seqforindrep}
	Let $w=t_1\cdots t_r$ such that $t_j$ are all distinct. Let $P_I=B^*W_IB^*$, where $I=\supp(w)$.  Set $I_u=\supp(u)$ for $u\preceq w$.  Then the sequence (\ref{stseq}) is an acyclic complex. Furthermore, $d_0$ is injective and
	the cokernel of $d_{\ell(w)-1}$ is $\Ind_{P_I(\Fq)}^{\GL_n(\Fq)}{\normalfont\text{St}_{L_I}}$, where
	\[
\St_{L_I}	:=\Ind_{(B^*\cap L_I)(\Fq)}^{L_I(\Fq)}\mathds{1}_{\mathbb Z/p^m\mathbb Z}\Big/ \sum_{(P\cap L_I)\supsetneq (B^*\cap L_I)} \Ind_{(P\cap L_I)(\Fq)}^{L_I(\Fq)}\mathds{1}_{\mathbb Z/p^m\mathbb Z}.
	\]
	\end{proposition}
\begin{proof}
We denote $\Ind_{H(\Fq)}^{G(\Fq)}\mathds{1}_{\mathbb Z/p^m\mathbb Z}$ by $i_H^G$ for any subgroup $H$ of a group $G$ when there is no ambiguity. 
	Recall that we have $P_I=L_I\ltimes U_I$, where $L_I$ is a Levi subgroup of $\GL_n$. In particular, $L_I$ is a reductive algebraic group over $\Fpbar$ defined over $\Fq$. Note that the Weyl group of $L_I$ is exactly $W_I$.  Since $L_I$ is reductive, by \cite[Theorem 1]{Solomon} cf. \cite[\textsection 7]{CLT},  the following sequence
		\begin{equation}
	\label{Levititscomplex}
	\bigoplus_{\substack{u\prec w\\ \ell(u)=\ell(w)-1}}
	i_{P_{I_u}\cap L_I}^{L_I}
\overset{d_{1}}{\longrightarrow} 
		\cdots \to \bigoplus_{\substack{u\prec w\\  \ell(u)=1}}
		i_{P_{I_{u}}\cap L_I}^{L_I}
		\overset{d_{\ell(w)-1}}{\longrightarrow} 
	 	i_{B^*\cap L_I}^{L_I}
\end{equation}
identifies with the combinatorial Tits complex $\Delta$ of $L_I$ tensored with $\mathbb Z/p^m\mathbb Z$. In particular,
\[
H_0(\Delta, \mathbb Z)=\mathbb Z \hspace{0.7cm} \text{\normalfont and}\hspace{0.7cm} H_{\ell(w)}(\Delta, \mathbb Z)=\mathbb Z^{\vert U_{L_I}(\Fq)\vert },
\]
where $U_{L_I}$ is a maximal $\Fq$-split unipotent subgroup of $L_I$, and $H_j(\Delta, \mathbb Z)=0$ otherwise. 
By the Universal Coefficients Theorem, we know that $\ker(d_1)=i_{L_I}^{L_I}$ and $\text{\normalfont coker}(d_{\ell(w)-1})=\St_{L_I}$.

We have the following identification 
\[
\Ind_{P_{I_u}(\Fq)}^{P_I(\Fq)}\mathds{1}_{\ZZ /p^m\ZZ}=\Ind_{(P_{I_u}\cap L_I)(\Fq)}^{L_I(\Fq)}\mathds{1}_{\ZZ /p^m\ZZ}
\]
for all $u\preceq w$.
Hence may rewrite the acyclic complex (\ref{Levititscomplex}) as follows: 

\begin{multline}\label{parabtitscomplex}
	0\to 	i_{P_I}^{P_I}\overset{}{\longrightarrow} \bigoplus_{\substack{u\prec w\\ \ell(u)=\ell(w)-1}}
	i_{P_{I_u}}^{P_I}
	\to
		\cdots \\
		\cdots \to 
	 	i_{B^*}^{P_I}
	 	\to \Ind_{B^*(\Fq)}^{P_I(\Fq)}\mathds{1}_{\mathbb Z/p^m\mathbb Z}\Big/ \sum_{P\supsetneq B^*} \Ind_{P(\Fq)}^{P_I(\Fq)}\mathds{1}_{\mathbb Z/p^m\mathbb Z}\to 0.
\end{multline}

 Recall that since $P_I(\Fq)$ is a finite subgroup of $\GL_n(\Fq)$, the functor $\Ind_{P_I(\Fq)}^{\GL_n(\Fq)}$ is exact. Thus we may apply the functor $\Ind_{P_I(\Fq)}^{\GL_n(\Fq)}$ to the acyclic complex (\ref{parabtitscomplex}) and obtain the complex (\ref{stseq}). Therefore the complex (\ref{stseq}) is acyclic and $d_{\ell(w)-1}$ has the cokernel as desired.  
	\end{proof}

We can also prove Proposition \ref{seqforindrep} algebraically. 
\begin{proof}[Alternative proof]
Denote the induced representations $\Ind_{H(\Fq)}^{G(\Fq)}\mathds{1}_{\mathbb Q}$ by $i_H^G(\mathbb Q)$ when there is no ambiguity. Set $A_k=i_{P_{u_k}\cap L_I}^{L_I}(\mathbb Q)$ with $k=1,\ldots,\ell(w)$ such that $u_k\prec w, \ell(u_k)=\ell(w)-1$ and $P_{u_k}=P_{\supp(u_k)}$. It is verified in \cite[Theorem 3.2.5]{DOR} that for any subsets $I,J\subseteq\{1,\ldots,\ell(w)\}$,
\begin{equation}\label{ScSte}
\left(\sum_{i\in I}A_i\right)\cap \left(\bigcap_{j\in J}A_j\right)=\sum_{i\in I}\left(A_i\cap\left(\bigcap_{j\in J}A_j\right)\right).	
\end{equation}

Note that loc.cit. proved this for generalized Steinberg representations, but it does apply for the scenario of the Steinberg representation itself. By \cite[Proposition 2.6]{ScSt}, one obtains an acyclic complex:
\begin{equation}\label{rationalSt}
0\to i_{L_I}^{L_I}(\mathbb Q)	\to
\bigoplus_{\substack{u\prec w\\ \ell(u)=\ell(w)-1}}
	i_{P_{I_u}\cap L_I}^{L_I}(\mathbb Q)
	\to
		\cdots \to 
		i_{B^*\cap L_I}^{L_I} (\mathbb Q)
	 	\to \St_{L_I}(\mathbb Q)\to 0,
\end{equation}
where $\St_{L_I}(\mathbb Q)$ is the Steinberg representation with coefficients in $\mathbb Q$. The complex (\ref{rationalSt}) gives a basis for $i_{B^*\cap L_I}^{L_I} (\mathbb Q)$ such that for all $P\supset B$, the  subrepresentation $i_{P\cap L_I}^{L_I} (\mathbb Q)$ is generated by a subset of this basis. 
More precisely, one starts with fixing a basis for $i_{L_I}^{L_I}(\mathbb Q)$ and inductively fix bases for each constituent of the next term in the complex. This ensures that the intersections would still be free modules generated by a basis element.
 
Taking the $\mathbb Z$-lattice with respect to this basis yields the sub-representations $i_{P\cap L_I}^{L_I} (\mathbb Z)$. In particular, the equality (\ref{ScSte}) holds after intersecting with the $\mathbb Z$-lattice and the intersections are free $\ZZ$-modules generated by basis elements. Thus we  can tensor this $\mathbb Z$-lattice with $\mathbb Z/p^m \mathbb Z$ to obtain  $ B_k:=i_{P_{u_k}\cap L_I}^{L_I}(\mathbb Z/p^m \mathbb Z)$, such that $B_k$'s satisfy the condition (\ref{ScSte}).
Thus by \cite[Proposition 2.6]{ScSt}, we obtain the acyclic complex as desired.

\end{proof}

\subsection{A spectral sequence associated to the stratification} 
Let $w$ be as in Section \ref{acyclicity}.
Denote the category of sheaves on $\Xo(w)_{\text{\'{e}t}}$ by $\normalfont\text{Sh}(\Xo(w)_{\text{\'{e}t}})$. For any closed subscheme $Z$ of $\Xo(w)$ with the inclusion map $\iota:Z\to \Xo(w)$, we denote $\iota_*\mathbb Z/p^m\mathbb Z$ by  $\left(\mathbb Z/p^m\mathbb Z\right)_{Z}$. Again, since $\iota_*$ is exact, we have an  isomorphism of $\mathbb Z/p^m\mathbb Z$-modules:
\[
\Het^r\left(Z ,\mathbb Z/p^m\mathbb Z\right)\overset{\sim}{\longrightarrow} H_\et^r\left(\Xo(w), \left(\mathbb Z/p^m\mathbb Z\right)_{Z}\right).
\]
In addition, if we assume that $Z$ is stable under $\GL_n(\Fq)$-action, then the above isomorphism is $\GL_n(\Fq)$-equivariant. 

\medskip
Let $R=\mathbb Z/p^m\mathbb Z$, for any $m> 0$. For $u\preceq w$, we denote the constant sheaf $R$ on $\Xo(u)$ by $R_{\Xo(u)}$.  
We have a sequence of constant sheaves on $\Xo(w)_{\text{\'{e}t}}$:
\begin{multline}\label{filtrationZpn}
	R_{\Xo(w)}
	\to \bigoplus_{\substack{u \prec w\\ \ell(u)=\ell(w)-1}} R_{\Xo(u)}\to \bigoplus_{\substack{u \prec w\\  \ell(u)=\ell(w)-2}} R_{\Xo(u)}\to \cdots \\
 \cdots \to \bigoplus_{\substack{u \prec w\\  \ell(u)=\ell(w)-i}} R_{\Xo(u)} \to \cdots \to \bigoplus_{\substack{u \prec w\\  \ell(u)=1}} R_{\Xo(u)} \to R_{X(e)}.
\end{multline}
Let $\{U_{\gamma}\to \Xo(w)\}_{\gamma}$ be an \'{e}tale cover of $\Xo(w)$, and consider the $i-1, i, i+1$-th terms of this complex 
\begin{equation*}\label{summand}
 \bigoplus_{\substack{u \prec w\\  \ell(u)=\ell(w)-i+1}} R_{\Xo(u)}(U_{\gamma})
 \overset{d_{i-1}}{\longrightarrow}
 \bigoplus_{\substack{u \prec w\\  \ell(u)=\ell(w)-i}} R_{\Xo(u)}(U_{\gamma})
 \overset{d_i}{\longrightarrow}
 \bigoplus_{\substack{u \prec w\\  \ell(u)=\ell(w)-i-1}} R_{\Xo(u)}(U_{\gamma}).
\end{equation*}
We now describe the maps $d_i$. 
Label each summand of the $i$-th term in the complex by $u_i$, by abuse of notation. 
Let

\[(f_{u_i})_{u_i}\in \bigoplus_{\substack{u \prec w\\  \ell(u)=\ell(w)-i}} R_{\Xo(u)}(U_{\gamma})\]
be a section, then we have   
	\[
	d_i\left((f_{u_i})_{u_i}\right)= \left(\sum_{u_i} (-1)^{\alpha(u_i\to u_{i+1})} \left.f_{u_i}\right\vert_{\Xo(u_{i+1})}\right)_{u_{i+1}} ,
	\]
	where $\alpha(u_i\to u_{i+1})$ is as defined in Section \ref{acyclicity}. Here the restriction of $f_{u_i}$ to $\Xo(u_{i+1})$ can be nonzero if and only if $\Xo(u_{i+1})\subseteq \Xo(u_i)$.

	When we fix $u_{i+1}\preceq u_{i-1}\preceq w$, as $\ell(u_{i-1})=\ell(u_{i+1})-2$, there are only two ways to take restrictions from $\Xo(u_{i-1})$ to $\Xo(u_{i+1})$ via $\Xo(u_i')$ for some $u_{i-1}\preceq u_i'\preceq u_{i+1}$.  Thus by the definition of the function $\alpha$, we may conclude that $d_i\circ d_{i-1}=0$. Therefore (\ref{filtrationZpn}) is a complex. 
\begin{lemma}\label{acyclic}
Let  $w=t_1\cdots t_r$ such that $t_j$ are all distinct. Then
the complex (\ref{filtrationZpn}) of sheaves on $\Xo(w)_{\normalfont\text{\'{e}t}}$ is acyclic. 	
\end{lemma}
\begin{proof}
	It suffices to check the acyclicity of the complex (\ref{filtrationZpn}) on the stalks. Let $x\in \Xo(w)$, then the complex would simplify depending on which closed subscheme $x$ lives in.  
	\begin{equation}\label{stalk}
	  \bigoplus_{\substack{u \prec w\\  \ell(u)=\ell(w)-i+1}} R_{\Xo(u),x}
 \overset{d_{i-1}}{\longrightarrow}
 \bigoplus_{\substack{u \prec w\\  \ell(u)=\ell(w)-i}} R_{\Xo(u),x}
 \overset{d_i}{\longrightarrow}
 \bigoplus_{\substack{u \prec w\\  \ell(u)=\ell(w)-i-1}} R_{\Xo(u),x}
	\end{equation}
	
	If we have $x\in X(u_{i-1})$ for some $u_{i-1}\prec w$, $\ell(u_{i-1})=\ell(w)-i+1$. Then we know that $x\notin \Xo(u)$ for all $u\preceq w$ with $\ell(u)\leq i$ and thus
	$R_{\Xo(u),x}=0$.  Then the complex \ref{stalk} is trivially exact at the $i$-th term. 
	
	If we have $x\in X(u_i)$ for some $u_i\prec w$, $\ell(u_i)=\ell(w)-i$. Then for all $R_{\Xo(u),x}=0$ for $u\prec w$, $\ell(u)\leq i-1$. The complex \ref{stalk} becomes
	
	\begin{equation*}
	 \bigoplus_{\substack{u_i\prec u \prec w\\  \ell(u)=\ell(w)-i+1}} R_{\Xo(u),x}
 \overset{d_{i-1}}{\longrightarrow}
 R_{\Xo(u_i),x}
 \overset{}{\longrightarrow}
0.
	\end{equation*}
Since $d_{i-1}$ is obviously surjective, this complex is exact at the $i$-th term in this case. 

If $x\in \Xo(u_{i+1})$,  then for any $f\in \K(d_i)$, all the summand of $d_i(f)$ are $0$. Now if $f=(f_t)_t$, then  for each summand of $d_i(f)$, there exists an even number of nonzero $f_t$'s that maps to it. 
Each such pair of $f_t, f_t'$ would have the property $f_t=-f_t'$ or $f_t=f_t'$. This is because all summands of $d_i(f)$ are $0$.  Now by the definition of $d_{i-1}$, we may build an element of the $i-1$-th term of below using the nonzero terms of $f_t$. 
	\begin{equation*}
	 \bigoplus_{\substack{u \prec w\\  \ell(u)=\ell(w)-i+1}} R_{\Xo(u),x}
 \overset{d_{i-1}}{\longrightarrow}
 \bigoplus_{\substack{u \prec w\\  \ell(u)=\ell(w)-i}} R_{\Xo(u),x}
 \overset{d_i}{\longrightarrow}
 \bigoplus_{\substack{u \prec w\\  \ell(u)=\ell(w)-i-1}} R_{\Xo(u),x}
	\end{equation*}
\end{proof}

\begin{corollary}\label{resolutionx(w)}
Let $j:X(w)\hookrightarrow \Xo(w)$ be the open immersion. Then the complex (\ref{filtrationZpn}) gives a resolution for $j_!\mathbb Z/p^m\mathbb Z$.	
\end{corollary}
\begin{proof}
We need to verify that the following complex is exact at $R$:
\[
0\to j_! R\overset{d_{-1}}{\longrightarrow} R\overset{d_0}{\longrightarrow}\bigoplus_{\substack{u\prec w\\\ell(u)=\ell(w)-1}} R_{\Xo(u)}.
\]
We verify on the stalks. When $x\in X(w)$, we have $j_! R_x\cong R_x$. When $x\in \Xo(w)\backslash X(w)$, then 
\[
R_x\overset{d_0}{\longrightarrow}\bigoplus_{\substack{u\prec w\\\ell(u)=\ell(w)-1}} R_{\Xo(u),x} 
\]
is injective. 
\end{proof}
	 Since the category $\normalfont\text{Sh}_\et(\Xo(w))$ has enough injective objects, by Corollary \ref{resolutionx(w)},  the complex (\ref{filtrationZpn})  is quasi-isomorphic to an injective resolution of $j_!\mathbb Z/p^m\mathbb Z$.  
Take the spectral sequence associated to the complex (\ref{filtrationZpn}), we have
\begin{equation}\label{ss}
E_1^{i,j}= \bigoplus_{\substack{ u\prec w\\ \ell(u)=\ell(w)-i}} H^j_
\et\left( \Xo(u), \mathbb Z/p^m\mathbb Z\right)	\Rightarrow H^{i+j}_{\et, c}\left(X(w),\mathbb Z/p^m\mathbb Z\right).
\end{equation}

\subsection{The \'{e}tale cohomology with compact support for $X(w)$ with coefficients in $\mathbb Z/p^m\mathbb Z$ and $\mathbb Z_p$}
\begin{theorem}\label{modpnX(w)}
Let $G=\GL_n$,  and $w=t_1\cdots t_r\in W$ such that $t_j\in S$ are distinct from one another.  Set $I=\supp(w), I_u=\supp(u)$ and $P_I=B^*W_IB^*=U_I\rtimes L_I$.
Then  
for $k\neq \ell(w), m>0$, 
\begin{equation*}
	H_{\et, c}^k\left(X(w),\mathbb Z/p^m\mathbb Z\right)=0,
\end{equation*}	
and
\begin{equation*}
	H^{\ell(w)}_{\et, c}\left(X(w),\mathbb Z/p^m\mathbb Z\right) \cong \Ind_{B^*(\Fq)}^{\GL_n(\Fq)}\mathds{1}_{\mathbb Z/p^m\mathbb Z}\Bigg/\sum_{\substack{u\prec w\\ \ell(u)=1}}\Ind_{P_{I_{u}}(\Fq)}^{\GL_n(\Fq)}\mathds{1}_{\mathbb Z/p^m\mathbb Z},
\end{equation*}
	 In particular, 
	 
	 \[H^{\ell(w)}_{\et, c}\left(X(w),\mathbb Z/p^m\mathbb Z\right)\cong \Ind_{P_I(\Fq)}^{\GL_n(\Fq)}{\normalfont\text{St}_{L_I}},
	  \]
	  where $\normalfont\text{St}_{L_I}$ is the Steinberg representation for $L_I(\Fq)$ with coefficients in $\mathbb Z/p^m\mathbb Z$. 
	 
\end{theorem}
\begin{proof} 
By Proposition \ref{ZpnZ}, we know that the $E_1^{i,0}$-terms of the spectral sequence \ref{ss} agree with the sequence \ref{stseq}. In addition, for all $j>0$ and $u\preceq w$, we have
	\[
	H^j\left(\Xo(u),\mathbb Z/p^m\mathbb Z\right)=0,
	\] 
	so there is no nonzero terms at $E^{i,j}_1$ when $j\neq 0$ and $E_2=E_\infty$.
It follows from Proposition \ref{seqforindrep} that 
\begin{equation*}
	 	E_2^{\ell(w),0} \cong \Ind_{B^*(\Fq)}^{\GL_n(\Fq)}\mathds{1}_{\mathbb Z/p^m\mathbb Z}\Bigg/\sum_{\substack{u\prec w\\ \ell(u)=1}}\Ind_{P_{I_{u}}(\Fq)}^{\GL_n(\Fq)}\mathds{1}_{\mathbb Z/p^m\mathbb Z},
	 \end{equation*}
	 and $E_2^{i,j}=0$ otherwise. The proof is thus concluded because $E_2^{i,0}\cong H^{i}_{\normalfont\text{\'{e}t, c}}\left(X(w),\mathbb Z/p^m\mathbb Z\right)$.

\end{proof}

\begin{corollary}\label{ZpX(w)}
Let $G=\GL_n$, and $w=t_1\cdots t_r\in W$ such that the $t_j$'s are distinct from one another. Denote $I=\supp(w), I_u=\supp(u), u\preceq w$ and $P_I=B^*W_IB^*$.
Then  
for $k\neq \ell(w), m>0$,
\begin{equation*}
	H_{\et, c}^k\left(X(w),\mathbb Z_p\right)=0,
\end{equation*}	
and
	 \begin{equation*}
	 	 H^{\ell(w)}_{\et, c}\left(X(w),\mathbb Z_p\right) \cong \Ind_{B^*(\Fq)}^{\GL_n(\Fq)}\mathds{1}_{\mathbb Z_p}\Bigg/\sum_{\substack{u\prec w\\ \ell(u)=1}}\Ind_{P_{I_{u}}(\Fq)}^{\GL_n(\Fq)}\mathds{1}_{\mathbb Z_p}.
	 \end{equation*}
	 In particular, 
	 
	 \[H^{\ell(w)}_{\et, c}\left(X(w),\mathbb Z_p\right) \cong \Ind_{P_I(\Fq)}^{\GL_n(\Fq)}{\normalfont\text{St}_{L_I}},\]
	  where $\normalfont\text{St}_{L_I}$ is the Steinberg representation for $L_I(\Fq)$ with coefficients in $\mathbb Z_p$. 
	 \end{corollary}
\begin{proof}
By Theorem \ref{modpnX(w)}, we have for all $k\neq \ell(w)$ and $m>0$, $H_{\et, c}^k\left(X(w),\mathbb Z/p^m\mathbb Z\right)=0$. Thus the tower $\left\{H_{\et, c}^k\left(X(w),\mathbb Z/p^m\mathbb Z\right)\right\}_m$ of abelian groups satisfy the Mittag--Leffler condition. Thus for all $k\neq \ell(w)$, we have
\[
H_{\et, c}^k\left(X(w),\mathbb Z_p\right)=\varprojlim_m H_{\et, c}^k\left(X(w),\mathbb Z/p^m\mathbb Z\right)=0.
\]
On the other hand, 
whenever we have  $m>l$ and a ${\rm mod}\ p^{l}$ map
\[\arraycolsep=1pt\def\arraystretch{1.5}
\begin{array}{rcl}
	\mathbb Z/p^m\mathbb Z &\longrightarrow &\mathbb Z/p^l\mathbb Z \\
	b\ {\rm mod} \ p^m &\longmapsto & b \ {\rm mod} \ p^l,
\end{array}
\]
there  is a short exact sequence of sheaves on $\Xo(w)_\et$:
\[
0\to \ZZ/p^{m-l}\ZZ \to \ZZ/p^{m}\ZZ\to \ZZ/p^{l}\ZZ\to 0.
\]
By  \cite[Corollary 8.14]{milneLEC}, since $j: X(w)\to \Xo(w)$ is an open immersion, we know that $j_!$ is an exact functor, so there is a short exact sequence 
\[
0\to j_!\ZZ/p^{m-l}\ZZ \to j_!\ZZ/p^{m}\ZZ\to j_!\ZZ/p^{l}\ZZ\to 0. 
\]
Taking the associated long exact sequence yields
\begin{multline*}
	\cdots\to H^{\ell(w)}_{\text{\'{e}t}}\left(\Xo(w), j_!\ZZ/p^{m}\ZZ \right)\to H^{\ell(w)}_{\text{\'{e}t}}\left(\Xo(w), j_!\ZZ/p^{l}\ZZ \right)\to\\
	\to  H^{\ell(w)+1}_{\text{\'{e}t}}\left(\Xo(w), j_!\ZZ/p^{m-l}\ZZ \right)\to \cdots.
\end{multline*}
The vanishing result for $k\neq \ell(w)$ from Theorem \ref{modpnX(w)} implies that 
\[
H^{\ell(w)}_{\et, c}\left(X(w), \ZZ/p^{m}\ZZ \right)\to H^{\ell(w)}_{\et, c}\left(X(w), \ZZ/p^{l}\ZZ \right)
\]
is surjective. 
Hence the tower of abelian groups $\left\{H_{\et, c}^{\ell(w)}\left(X(w),\mathbb Z/p^m\mathbb Z\right)\right\}_m$ satisfies the Mittag--Leffler condition.  
	Therefore
	\[
	H^{\ell(w)}_{\et, c}\left(X(w),\mathbb Z_p\right)=\varprojlim_m H^{\ell(w)}_{\et, c}\left(X(w),\mathbb Z/ p^m\mathbb Z \right),
	\] 
	and
 \begin{equation*}
	 	 H^{\ell(w)}_{\et, c}\left(X(w),\mathbb Z_p\right) \cong \Ind_{B^*(\Fq)}^{\GL_n(\Fq)}\mathds{1}_{\mathbb Z_p}\Bigg/\sum_{\substack{u\prec w\\ \ell(u)=1}}\Ind_{P_{I_{u}}(\Fq)}^{\GL_n(\Fq)}\mathds{1}_{\mathbb Z_p}.
	 \end{equation*}
\end{proof}

\begin{corollary}
 Let $G=\GL_n$ and $w\in F^+$. Let $v\in F^+$ such that $\supp(v)=\supp(w)$ and $v=s_{\alpha_1}\cdots s_{\alpha_{r}}\in W$ with $s_{\alpha_t}$ all distinct. Let $R=\mathbb Z/p^m\mathbb Z$ or $\mathbb Z_p$, $m>0$. Set $I=\supp(w), I_u=\supp(u),$ and $P_I=B^*W_IB^*$.
Then  
for $k\neq \ell(w)$,
\begin{equation*}
	H_{\et, c}^k\left(X(w),R\right)=0
\end{equation*}	
and
	 \begin{equation*}
	 	 H^{\ell(w)}_{\et, c}\left(X(w),R \right)\cong H^{\ell(v)}_{\et, c}\left(X(v),R \right) \cong \Ind_{B^*(\Fq)}^{\GL_n(\Fq)}\mathds{1}_R\Bigg/\sum_{\substack{u\prec v\\ \ell(u)=1}}\Ind_{P_{I_{u}}(\Fq)}^{\GL_n(\Fq)}\mathds{1}_R.
	 \end{equation*}	 
	 In particular, 
	 
	 \[H^{\ell(w)}_{\et, c}\left(X(w),R\right)\cong  \Ind_{P_I(\Fq)}^{\GL_n(\Fq)}{\normalfont\text{St}_{L_I}}, \]
	  where $\normalfont\text{St}_{L_I}$ is the Steinberg representation for $L_I(\Fq)$ with coefficients in $R$.

\end{corollary}
\begin{proof}
 Analogous to \cite[Proposition 2.11]{Orlik18}, by induction on $\ell(w)-\ell(v)$, the complex (\ref{stseq}) we get for $w\in F^+$ is homotopic to the complex (\ref{stseq}) for any $v\in F^+$ such that $ \supp(v)=\supp(w)$. The rest follows from Theorem \ref{modpnX(w)} and Corollary \ref{ZpX(w)}.
\end{proof}

\appendix
\section{$F$-singularities and pseudo-rational singularities of $\overline{X(w)}$}\label{AppendixA}

In \cite[Theorem 2.2]{LRT}, it was shown that Schubert varieties are globally $F$-regular. As already seen in \cite[\textsection 4.3]{Lusztig81}, certain local properties of Schubert varieties can be carried over to local properties of the Zariski closure of Deligne--Lusztig varieties $\overline{X(w)}$ via the Lang map. This allows us to show that  $\overline{X(w)}$ has pseudorational singularities. There is a generalization of the result of \cite[Theorem 3.2.8]{CR11} relaxing the requirement of smoothness to pseudo-rational singularity \cite[Theorem 8.13]{Kovacs}. This allows us to obtain the result in Proposition \ref{grouprel} for all connected reductive $G$ defined over a finite field $\Fq$.

Throughout this section, let $G$ be an arbitrary reductive group as defined in Section \ref{Gdefn}.

\subsection{Review of definitions}
We recall some relevant definitions of $F$-singularities of commutative rings and schemes in positive characteristic. More detailed discussions can be found in \cite[Chapter 3]{Smith19}, for example.

\begin{dfn}[{\cite[p.14, Definition 1.17]{Smith19}}]
\normalfont
	Let $R$ be a ring of characteristic $p>0$ and $F:R\to R$ be the $p$-Frobenius morphism. We say that $R$ is \emph{$F$-finite} if $F$ is a finite map. 
\end{dfn}
By \cite[p.47]{Smith19}, finitely generated algebras over a perfect field as well as their localizations and completions are always $F$-finite. Thus if $Y$ is a scheme locally of finite type over a perfect field, then the local rings $\Osh_{Y, y}$ for all $y\in Y$ are $F$-finite.

\begin{dfn}[{\cite[p.128]{HH89}}]
\normalfont
	Let $R$ be a $F$-finite ring. If for every $r\in R$ not contained in any minimal prime ideals of $R$, 
	there exists some positive integer $e$ such that the map of $R$-modules $R\to R^{1/{p^e}}$ given by $1\mapsto r^{1/{p^e}}$ splits, 
	then we say that $R$ is \emph{strongly $F$-regular}.
\end{dfn}
\begin{dfn}
\normalfont
Let $Y$ be a locally of finite type in characteristic $p>0$ with geometric Frobenius $F$. We say that $Y$ is \emph{strongly $F$-regular} if the local ring $\Osh_{Y,y}$ is strongly $F$-regular for all $y\in Y$.
\end{dfn}

The property of being strongly $F$-regular is indeed a local condition by \cite[p. 56, Proposition 4.12]{Smith19} cf. \cite[Theorem 5.5]{HH94}. Note that in the literature, the strongly $F$-regular property for schemes may also be called \emph{locally $F$-regular} in order to distinguish from the \emph{globally $F$-regular} property, which is defined via the section ring of an ample line bundle \cite{Smith00}. 

\begin{dfn}
\normalfont
	Let $Y$ be a locally of finite type in characteristic $p>0$ with geometric Frobenius $F$.  Then $Y$ is said to be \emph{ $F$-rational} when the local ring $\Osh_{Y,y}$ is $F$-rational for all $y\in Y$.
\end{dfn}

The definition of \emph{$F$-rationality} for a ring of characteristic $p>0$ involves more definitions, which we will not directly work with, so we refer to \cite[Definition 1.3]{Smith97} for its precise statement. 
It then follows from the definitions that strongly $F$-regular rings are $F$-rational.

\begin{remark}
	In the literature, the notion of \emph{excellent} rings and schemes are sometimes used to make more general statements. We shall remark that   locally of finite type $\Fpbar$-schemes are excellent by \cite[Prop 7.8.6, p. 217]{EGA42}. As localization of local sections of excellent schemes remain to be excellent rings, all the local rings we work with in this section are in fact excellent rings. 
\end{remark}

\subsection{Cohomology of $\overline{X(w)}$}
For any $w\in W$, the Zariski closure $\overline{B^*wB^*}$  of a Bruhat cell in $G$ are \emph{large Schubert varieties} in the sense of \cite[p. 956]{BrionThomsen}. We shall use this fact to show the Zariski closure of Deligne--Lusztig varieties have pseudo-rational singularities. As before, we work in the more general setting with generalized Deligne--Lusztig varieties.

Recall that we have the natural quotient map $\alpha: F^+\to W$. Let $w_1,w_2\in F^+$ and $v\in F^+$ such that $v$ corresponds to a reduced expression of an element in $\alpha(v)\in W$. Write $w_1=s_{i_1}\cdots s_{i_{l(w_1)}}$ and  $w_2=s_{i_{l(w_1)+2}}\cdots s_{i_{l(w_1)+l(w_2)+1}}$. We may then fix a notation:

\begin{multline*}
	\Xo(w_1\underline{v}w_2):=\left\{(B_0,\ldots,B_{l(w_1)+l(w_2)+1})\in X^{l(w_1)+l(w_2)+2}\right\vert \\  
	 (B_{j-1}, B_j)\in \overline{O(s_{i_j})},j=1,\ldots, i_{l(w_1)},
	\\
	 (B_{j-1}, B_j)\in \overline{O(s_{i_j})}, j=l(w_1)+2, \ldots, l(w_1)+l(w_2)+1,
	\\
	\left.
	(B_{l(w_1)}, B_{l(w_1+1)})\in \overline{O(v)}, B_{l(w_1)+l(w_2)+1}=FB_0
	 \right\}.
\end{multline*}
Note that if $w_1$ and $w_2$ are trivial, then $\Xo(w_1\underline{v}w_2)=\overline{X(v)}$.
\begin{dfn}\label{partialfrob}
\normalfont
For any integer $r>0$, define a partial Frobenius endomrophism $F_1: G^{r+1}\to G^{r+1}$ on the $r+1$-fold fibre product of $G$ over $\Fpbar$ by
\[
F_1(g_0,\ldots,g_r):=(g_1,\ldots,g_r, F(g_0)).
\]
\end{dfn}

Observe that $F_1^{r+1}$ is the Frobenius endomorphism of $G^{r+1}$, and that $G^{r+1}$ is again connected reductive. Denote the associated Lang map of $F_1$ by $\mathcal L_1$, which is defined by $h\mapsto h^{-1}F_1(h)$ for all $h\in G^{r+1}(\Fpbar)$. It follows from the Lang--Steinberg theorem that $\mathcal L_1$ is surjective. 
By \cite[Proposition 2.3.3]{DMR}, we get an isomorphism
\begin{multline}\label{w1vw2}
	\Xo(w_1\underline{v}w_2)\overset{\sim}{\longrightarrow}
	\left\{
	h\in G^{r+1}\right\vert
 \mathcal L_1(h)\in
 \\	\left. \overline{B^*\dot{s_1}B^*}\times \cdots\times \overline{B^*\dot{s_{i_{l(w_1)}}}B^*}\times \overline{B^*\dot{v}B^*}\times  \overline{B^*\dot{s_{l(w_1)+2}}B^*}\times\cdots \right.\\ \left.\cdots \times \overline{B^*\dot{s_{l(w_1)+l(w_2)+1}}B^*}
	\right\}\big/(B^*)^{r+1},
\end{multline}
where $r=l(w_1)+l(w_2)+1$.

\begin{proposition}\label{locallystrongly}
	The scheme $\Xo(w_1\underline{v}w_2)$ for $w_1, w_2\in F^+$ and $v\in F^+$ being a reduced expression for an element in $W$, is strongly $F$-regular. 
\end{proposition}
\begin{proof}
	By  \cite[p. 958, Corollary 4.1]{BrionThomsen}, for all $w\in W$, $\overline{B^*\dot{w}B^*}$ is strongly $F$-regular. As strongly $F$-regularity is preserved under faithfully flat descent \cite[Theorem 5.5]{HH94}, we see that the product
	\[
	\overline{B^*\dot{s_1}B^*}\times \cdots\times \overline{B^*\dot{s_{i_{l(w_1)}}}B^*}\times \overline{B^*\dot{v}B^*}\times  \overline{B^*\dot{s_{l(w_1)+2}}B^*}\times\cdots\times  \overline{B^*\dot{s_{l(w_1)+l(w_2)+1}}B^*}
	\]
	is again strongly $F$-regular. Denote this product by $Z_{v}$, and we have the following diagram
	\begin{equation}
\begin{tikzcd}
\Xo(w_1\underline{v}w_2) & \mathcal{L}_1^{-1}\left(Z_v\right)\arrow[r, "\mathcal L_1"]\arrow[l,"\pi"'] 
& Z_v.
\end{tikzcd}		
	\end{equation}

	Since $\mathcal L_1$  is \'{e}tale, the fibres of $\mathcal L_1$ are smooth. Thus  $\mathcal{L}_1^{-1}\left(Z_v\right)$ is strongly $F$-regular by  \cite[Theorem 3.6]{Aberbach01}.
We apply to $\pi$ the faithfully flat descent of strong $F$-regularity \cite[Theorem 5.5]{HH94}. 
Therefore $\Xo(w_1\underline{v}w_2)$ is strongly $F$-regular. 	
\end{proof}

\begin{remark}
(i)
The argument above is analogous to the argument in \cite[\textsection 4.3]{Lusztig81}. In particular, Proposition \ref{locallystrongly} implies that $\Xo(w_1\underline{v}w_2)$ are normal and Cohen--Macaulay cf. \cite[Proposition 2.3.5]{DMR}. The reason is that $F$-rational excellent local rings are normal and Cohen--Macaulay \cite[Theorem 4.2, Theorem 6.27]{HH94}.

(ii)
Instead of using the result that the Zariski closure of Bruhat cells are strongly $F$-regular, one may alternatively use the fact that the Schubert varieties $\overline{B^*\dot{w}B^*}/B^*$ are globally $F$-regular, which is shown in  \cite[Theorem 2.2]{LRT} and \cite[Theorem 8]{H06}. 

The global $F$-regularity of $\overline{B^*\dot{w}B^*}/B^*$ implies that it is strongly $F$-regular \cite[p. 75, Proposition 6.22]{Smith19}. Consider the quotient map $\pi: \overline{B^*\dot{w}B^*} \to \overline{B^*\dot{w}B^*}/B^*$. Since $\pi$ is a Zariski locally trivial fibre bundle with all fibres isomorphic to $B^*$,  it is a faithfully flat morphism. We may apply faithfully flat descent of strong $F$-regularity \cite[Theorem 5.5]{HH94} because $B^*$ is a smooth $\Fpbar$-scheme. Thus $\overline{B^*\dot{w}B^*}$ is strongly $F$-regular. The rest of the proof is the same as the proof of Proposition \ref{locallystrongly}.

\end{remark}

\begin{proposition}
\label{cohequiv}
	For any $w\in F^+$ of the form $w=w_1vw_2$, where $w_1,w_2,v\in F^+$ and $v$ has the same expression with a reduced expression of $\alpha(v)\in W$. Let $v'\in F^+$ be another reduced expression of $\alpha(v)$, then we have the following isomorphisms of cohomology for structure sheaves and canonical bundles: 
	\[
	H^k\left(\Xo(w_1vw_2),\Osh_{\Xo(w_1vw_2)}\right)\overset{\sim}{\longrightarrow}H^k\left(\Xo(w_1v'w_2),\Osh_{\Xo(w_1v'w_2)}\right)
	\]
	and 
	\[
	H^k\left(\Xo(w_1vw_2),\omega_{\Xo(w_1vw_2)}\right)\overset{\sim}{\longrightarrow}H^k\left(\Xo(w_1v'w_2),\omega_{\Xo(w_1v'w_2)}\right),
	\]
	for all $k\geq 0$.
\end{proposition}

\begin{proof}
For schemes locally of finite type over $\Fpbar$, the local rings being strongly $F$-regular implies them being pseudo-rational \cite[Theorem 3.1]{Smith97}. 
	Then by Proposition \ref{locallystrongly}, we know that the scheme $\Xo(w_1\underline{v}w_2)$ has pseudo-rational singularities cf. \cite[Definition 9.4]{Kovacs}. Furthermore, the $\Fpbar$-schemes $\Xo(w_1vw_2)$ and $\Xo(w_1v'w_2)$ are regular, so they also have pseudo-rational singularities. 
	
	By assumption, we know that $v$ and $v'$ have expressions corresponding to two different reduced expressions of $\alpha(v)$. Thus $\overline{O(v)}$ is isomorphic to $\overline{O(v')}$. Furthermore, $\overline{O}(v)$ and $\overline{O}(v')$ give two smooth compactifications of $\overline{O(v)}$. Hence we have the following cartesian squares via projections
	\begin{equation*}
\begin{tikzcd}
	\Xo(w_1vw_2) \arrow[r]\arrow[d,"f"'] &\overline{O}(v)\arrow[d]\\
	\Xo(w_1\underline{v}w_2)\arrow[r] & \overline{O(v)},
\end{tikzcd}
\end{equation*}
and
\begin{equation*}
\begin{tikzcd}
	\Xo(w_1v'w_2) \arrow[r]\arrow[d,"f'"'] &\overline{O}(v')\arrow[d]\\
	\Xo(w_1\underline{v}w_2)\arrow[r] & \overline{O(v)},
\end{tikzcd}
\end{equation*}
where $f$ and $f'$ are proper birational morphisms of $\Fpbar$-schemes. Via  \cite[Theorem 8.13]{Kovacs} we see that $\Xo(w_1vw_2)$ (resp. $\Xo(w_1v'w_2)$) has isomorphic cohomology groups for the structure sheaf and the canonical bundle with $\Xo(w_1\underline{v}w_2)$ via $f$ (resp. $f'$). 
	Therefore we have 
		\[
	H^k\left(\Xo(w_1vw_2),\Osh_{\Xo(w_1vw_2)}\right)\overset{\sim}{\longrightarrow}H^k\left(\Xo(w_1v'w_2),\Osh_{\Xo(w_1v'w_2)}\right)
	\]
	and 
	\[
	H^k\left(\Xo(w_1vw_2),\omega_{\Xo(w_1vw_2)}\right)\overset{\sim}{\longrightarrow}H^k\left(\Xo(w_1v'w_2),\omega_{\Xo(w_1v'w_2)}\right),
	\]
	for all $k\geq 0$.
\end{proof}

\begin{remark}
As an upshot, Proposition \ref{cohequiv} implies Proposition \ref{grouprel} if we set $v=sts$ and $v'=tst$ (resp. $v=st$ and $v'=ts$) whenever $sts=tst$ (resp. $st=ts$) in $W$.  Furthermore, Proposition \ref{locallystrongly} and \ref{cohequiv} work for any connected reductive group $G/\Fpbar$ defined over $\Fq$.
\end{remark}